\documentclass[12pt,reqno]{NumPDEsArticle}

\usepackage{fullpage}
\usepackage{hyperref}
\usepackage{amsmath,amssymb,amsthm}
\usepackage{todonotes}
\usepackage{nameref}
\usepackage{biblatex}
\usepackage{booktabs}
\usepackage{bm}
\usepackage{moreverb}
\usepackage{mathtools, mathrsfs}
\allowdisplaybreaks
\usepackage[shortlabels]{enumitem}
\usepackage{tikz}
\usepackage{tkz-euclide}
\usepackage[caption=false]{subfig}

\usetikzlibrary{shapes.geometric, calc}
\usepackage{xcolor}
\definecolor{TUblue}{rgb}{0,0.2157,0.4235}
\definecolor{TUgreen}{rgb}{0,0.3412,0.1725}
\usepackage{pgfplots}
\usepackage{pgfplotstable}
\pgfplotsset{%
	compat=newest,%
	every axis/.style={scale only axis},%
	grid style={densely dotted, semithick},%
}

\usepackage{NumPDEsMacros}

\newcommand{\exact}{\star}
\newcommand{\kk}{{\underline{k}}}

\definecolor{ferngreen}{rgb}{0.31, 0.47, 0.26}
\definecolor{HUblue}{rgb}{0,0.2157,0.4235}
\definecolor{bleu}{rgb}{0.19, 0.55, 0.91}

\def\cps{C_{{\rm OL}}}
\def\cqu{C_{\rm qu}}
\def\crel{C_{\rm rel}}
\def\cs{C_{{\rm ML}}}
\def\cscs{C_{\rm SCS}}
\def\cscsunif{\widehat C_{\rm SCS}}
\def\csdsqu{C^2_{{\rm SD}}}
\def\csd{C_{{\rm SD}}}
\def\qctr{q_{\rm ctr}}

\def\revision#1{{\color{red}#1}}
\def\P{\mathbb{P}}
\def\S{\mathbb{S}}
\def\V{\mathbb{V}}
\def\Vp{\mathbb{V}^p}

\newcommand\etalg{\zeta_{L}}
\newcommand\etdisc{\eta_{L}}
\def\coarse{h}

\def\reff#1#2{\stackrel{\eqref{#1}}{#2}}
\def\set#1#2{\big\{#1 \,:\, #2 \big\}}
\def\RR{\mathbb{R}}
\def\N{\mathbb{N}}
\def\dx{{\rm d} {x}}
\def\A{\mathbb{A}}
\def\T{\mathbb{T}}

\DeclareMathOperator*{\esssup}{ess\,sup}
\DeclareMathOperator*{\essinf}{ess\,inf}
\DeclareMathOperator*{\cost}{cost}
\def\enorm#1{|\!|\!| #1 |\!|\!|}
\renewcommand{\enorm}[2][]{#1|\!#1|\!#1| #2 #1|\!#1|\!#1|}
\def\norm#1#2{\|#1\|_{#2}}
\renewcommand{\norm}[3][]{#1\| #2 \|_{#3}}
\def\f{\boldsymbol{f}}
\def\K{\boldsymbol{K}}
\def\MM{\mathcal{M}}

\def\QQ{\mathcal{Q}}
\def\TT{\mathcal{T}}
\def\UU{\mathcal{U}}
\def\VV{\mathcal{V}}
\title{\textit{\MakeLowercase{hp}}-robust multigrid solver on locally refined meshes\\ for FEM discretizations of symmetric elliptic PDE\MakeLowercase{s}} 

\thanks{The authors thankfully acknowledge support by the Austrian Science Fund (FWF) through the SFB \emph{Taming complexity in partial differential systems} (grant SFB F65) and the standalone project \emph{Computational nonlinear PDEs} (grant P33216). The Vienna School of Mathematics supports Julian Streitberger.}
\author{Michael Innerberger}
\author{Ani Mira{\c c}i}
\author{Dirk Praetorius}
\author{Julian Streitberger}

\address{TU Wien, Institute of Analysis and Scientific Computing, Wiedner Hauptstr. 8-10/E101/4, 1040 Vienna, Austria}
\email{michael.innerberger@asc.tuwien.ac.at}
\email{ani.miraci@asc.tuwien.ac.at}% \rm(corresponding author)
\email{dirk.praetorius@asc.tuwien.ac.at}
\email{julian.streitberger@asc.tuwien.ac.at \quad \rm(corresponding author)}

\begin{document}

	\maketitle

	\begin{abstract}
	In this work, we formulate and analyze a geometric multigrid method for the iterative solution of the discrete systems arising from the finite element discretization of symmetric second-order linear elliptic diffusion problems. We show that the iterative solver contracts the algebraic error robustly with respect to the polynomial degree $p \ge 1$ and the (local) mesh size $h$. We further prove that the built-in algebraic error estimator which comes with the solver is $hp$-robustly equivalent to the algebraic error. The application of the solver within the framework of adaptive finite element methods with quasi-optimal computational cost is outlined. Numerical experiments confirm the theoretical findings.
	\end{abstract}

	\bigskip


	\bigskip

	
	%%%%%%%%%%%%%%%%%%%%%%%%%%%%%%%%%%%%%%%%%%%%%
	%%%%%%%%%%%%%%%%%%%%%%%%%%%%%%%%%%%%%%%%%%%%%

	\section{Introduction}\label{section:introduction}
	
	Numerical schemes for PDEs aim at approximating the solution $u^{\star}$ of the weak formulation with an error below a certain tolerance at minimal computational cost. Since the accuracy is spoiled by singularities, e.g., in given data or domain geometry, adaptive finite element methods (AFEMs) employ the loop
	\begin{equation*}
		\begin{tikzpicture}
			\tikzset{myline/.style={draw=lightgray, line width=1.5pt, rounded corners}}
			\tikzset{afemnode/.style={myline, fill=lightgray!50, minimum width={width("ESTIMATE")+2pt}}}
			\node[afemnode] (solve) at (0,0) {\texttt{SOLVE}};
			\node[afemnode,right=of solve] (estimate) {\texttt{ESTIMATE}};
			\node[afemnode,right=of estimate] (mark) {\texttt{MARK}};
			\node[afemnode,right=of mark] (refine) {\texttt{REFINE}};
			
			\tikzset{connection/.style={myline, gray, -stealth, , opacity=0.5}}
			\draw[connection] (solve) -- (estimate);
			\draw[connection] (estimate) -- (mark);
			\draw[connection] (mark) -- (refine);
			\draw[connection, dashed] (estimate.north) -- +(0,10pt) -| ($(solve.north west)!0.66!(solve.north east)$);
			\draw[connection] (refine.north) -- +(0,15pt) -| ($(solve.north west)!0.33!(solve.north east)$);
			
			\node at ($(solve.east)!0.5!(estimate.west)+(5pt,12pt)$) {\scriptsize adaptive stopping};
			\node at ($(estimate.east)!0.5!(mark.west)+(30pt,17pt)$) {\scriptsize solution accurate enough?};
		\end{tikzpicture}
	\end{equation*}
	to obtain a sequence of meshes $\TT_L$ that resolve such singularities.
	For a large class of problems, it is known that AFEM is \emph{rate-optimal}, i.e., one can construct an estimator $\eta_L(u_L^\star)$ from the exact Galerkin solution $u_L^\star$ for the discretization error $\enorm{u^\star - u_L^\star}$ that decreases with the largest possible rate with respect to the number elements in $\TT_L$; see, e.g., the seminal works~\cite{Dorfler_marking_96, Morin_data_osc_00, Bin_Dah_afem_conv_rat_04, Stev_opt_FE_07, Casc_Kreu_Noch_Sieb_08} or the abstract overview~\cite{Cars_Feis_Page_Prae_ax_adpt_14} for $h$-adaptive FEM with fixed polynomial degree $p$.
	%or the recent work~\cite{cnsv2017} for $hp$-adaptive FEM.
	
	% rate-optimal vs optimal costs
	In practice, 
	%however, 
	the \texttt{SOLVE} module may become computationally expensive (in contrast to all other modules) when employing a direct solver; see, e.g., \cite{Pfeiler_Praetorius_20, GHPS_21, MooAFEM} for a discussion of implementational aspects.
	Thus, usually, an iterative solver is employed to compute an approximation $u_L$ of $u_L^\star$ on each level, and the exact Galerkin solution $u_L^\star$ is not available.
	The question of whether the approximations $u_L$ converge with optimal rate with respect to the \emph{overall computational cost} 
	%(that is generically proportional to $\sum_{\ell=0}^L \#\text{(solver steps)} \cdot \#\TT_\ell$) 
	was already treated in the seminal work~\cite{Stev_opt_FE_07} under some \emph{realistic} assumptions about an abstract iterative solver.
	The recent work~\cite{GHPS_21} employs nested iterations and an adaptive stopping criterion to steer a uniformly contractive iterative solver, linking the \texttt{SOLVE} and \texttt{ESTIMATE} module in the above scheme by an inner loop.
	Then, it is shown that even the \emph{full sequence of iterates} converges with optimal rates with respect to the overall computational cost.
	For this reason, the design of algebraic solvers that are uniformly contractive and robust with respect to the discretization parameters is of utmost importance.
	
	The hierarchical structure of AFEM and the very nature of the arising linear systems suggest to use a \emph{multilevel} solver; see, e.g.,
	\cite{Hack_MG_book_85, Bra_McC_Rug_85, Bra_Pasc_Sch_substr_I_86, , Ban_Dup_Yse_88, Bra_Pas_Xu_90,  Zhang_multilevel_92, Rud_FAM_93, Oswald_book_94}. Different adaptive methods integrating a multilevel solver are possible; see, e.g., \cite{Bai1987} for generating local meshes, and \cite{Ruede1993} for a fully adaptive multigrid method that steers the local refinement process. In the context of AFEM, the adaptively constructed hierarchy of locally refined meshes calls for suitable \emph{local} solvers.
	We refer to~\cite{CNX_12} for a multilevel preconditioner on a mesh hierarchy consisting of one bisection in each step and \cite{Hipt_Wu_Zheng_cvg_adpt_MG_12,Wu_Zheng_Multigrid} for multiplicative multigrid methods, all of which are robust with respect to the mesh size $h$.
	Though these works allow for higher-order FEM, an analytic and numerical study on the behavior with increasing polynomial degree was not presented.
	This aspect is treated, e.g., in~\cite{Pavarino_94, Sch_Mel_Pec_Zag_08, Ant_Mas_Ver_pMG_18, Brubeck_Farrell_hoFEM_21}, which design iterative solvers that are robust with respect to the polynomial degree $p$ on various types of polyhedral meshes.
	The recent own work~\cite{Mir_Pap_Voh_lam_21} proposes a $p$-robust geometric multigrid which comes with a \emph{built-in algebraic error estimator} $\zeta_L(u_L)$, which is suited perfectly for \textsl{a posterori} steering (i.e., adaptive termination) of the algebraic solver.
	However, the employed patchwise smoothing associated to every vertex and every level causes a linear dependence on the number of adaptive mesh levels $L$.

	% contributions and main results
	In the present work, we modify the solver from \cite{Mir_Pap_Voh_lam_21} and overcome this dependence for locally refined meshes: we only apply \emph{local} lowest-order smoothing on patches which change in the refinement step on intermediate levels, whereas a patchwise (and hence parallelizable) higher-order smoothing on all patches of the finest level is applied.
	This solver only needs one post-smoothing step, requires no symmetrization of the procedure (see also \cite{DiPietro_Huelsemann_Matalon_Mycek_Ruede_Ruiz_21_HHO_MG}), and, in particular, has no tunable parameters since it utilizes optimal step-sizes on the error-correction stage.
	As the main result of the present work, we show that the proposed solver \emph{uniformly contracts} the algebraic error $\enorm{u_L^\star - u_L}$.
	Moreover, it comes with a built-in estimator $\zeta_L(u_L)$, which is shown to be \emph{equivalent} to 
	%the algebraic error 
	$\enorm{u_L^\star - u_L}$. Throughout, all involved estimates are robust in the discretization parameters $h$ and $p$. 
	
	As one potential application, we formulate an AFEM algorithm in the spirit of \cite{GHPS_21} that naturally embeds the multigrid solver and leverages the solver's built-in algebraic error estimator $\zeta_L(u_L)$ to stop the solver as soon as the discretization and algebraic error are comparable.
	%, i.e.,
	%\begin{equation*}
	%	\zeta_L(u_L) \leq \mu \, \eta_L(u_L)
	%	\quad \text{for some stopping parameter } \mu > 0,
	%\end{equation*}
	%where $\eta_L(u_L)$ denotes the residual error estimator evaluated for the inexact discrete solution $u_L$. 
	Adapting the arguments of~\cite{GHPS_21}, we prove that, for fixed polynomial degree $p$, the AFEM algorithm guarantees optimal convergence rates with respect to overall computational cost.
	%As the main result of the present work, we show that the solver \emph{uniformly contracts} the algebraic error $\enorm{u_L^\star - u_L}$, which is \emph{equivalent} to the built-in estimator $\zeta_L(u_L)$; all involved estimates are robust in the discretization parameters $h$ and $p$.
	%In particular, the proposed robust local multigrid solver satisfies the conditions of~\cite{GHPS_21}, which allows us to conclude that the entire sequence of our adaptive algorithm achieves optimal convergence rates with respect to overall computational cost.
	
	Using the open-source object-oriented 2D {\sc Matlab} code MooAFEM~\cite{MooAFEM}, we present
	%Finally, we present 
	a detailed numerical study of both the algebraic solver and the adaptive algorithm, including higher-order experiments and jumping coefficients.
	%The experiments are implemented in the open-source object oriented 2D {\sc Matlab} code MooAFEM~\cite{MooAFEM}.

	%structure of the paper
	The outline of this paper reads as follows: Section~\ref{section:Setting} first poses the model problem and introduces some notation. Then, we state
	%There, we also state our adaptive algorithm together with 
	the proposed multigrid solver (Algorithm~\ref{algorithm:solver}) and formulate our main results on $hp$-robust contraction (Theorem~\ref{theorem:solver}) and algebraic error control (Corollary~\ref{corollary:solver}).
	As a potential application, Section~\ref{subsection:AFEM} formulates an AFEM algorithm (Algorithm~\ref{algorithm:afem}) which employs nested iteration and an adaptive stopping criterion for the iterative solver using the built-in \textsl{a~posteriori} estimator for the algebraic error. Theorem~\ref{theorem:optimal-cost} proves optimal computational complexity of the proposed AFEM algorithm.
	%Section~\ref{section:main_theorems} then presents the main results of this work: $hp$-robust contraction of the multigrid solver and equivalence of the built-in error estimator to the algebraic error.
	%Together, this ultimately allows us to state the optimal computational complexity of the proposed AFEM algorithm.
	After we confirm the theoretical results by numerical examples in Section~\ref{section:numerical_experiments}, we present proofs of the main results in Section~\ref{section:proofs}.
	For better readability, we precede these proofs with three subsections presenting their core arguments: geometric properties of the meshes $\TT_L$, an $hp$-robust stable decomposition combining a local lowest-order multilevel stable decomposition from~\cite{Wu_Zheng_Multigrid} with a one-level $p$-robust decomposition from~\cite{Sch_Mel_Pec_Zag_08}, and a strengthened Cauchy--Schwarz inequality in the spirit of~\cite{CNX_12,Hipt_Wu_Zheng_cvg_adpt_MG_12}.
	
	%%%%%%%%%%%%%%%%%%%%%%%%%%%%%%%%%%%%%%%%%%%%%
	%%%%%%%%%%%%%%%%%%%%%%%%%%%%%%%%%%%%%%%%%%%%%

	%%%%%%%%%%%%%%%%%%%%%%%%%%%%%%%%%%%%%%%%%%%%%
	%%%%%%%%%%%%%%%%%%%%%%%%%%%%%%%%%%%%%%%%%%%%%
	%\section{Setting and notation}\label{section:Setting}
	\section{\textit{hp}-robust multigrid solver}\label{section:Setting}
	%%%%%%%%%%%%%%%%%%%%%%%%%%%%%%%%%%%%%%%%%%%%%
	%%%%%%%%%%%%%%%%%%%%%%%%%%%%%%%%%%%%%%%%%%%%%
	
	In this section, we formulate the model problem, the proposed geometric multigrid method, and the main results, while the proofs are postponed to Section~\ref{section:proofs}.

	%%%%%%%%%%%%%%%%%%%%%%%%%%%%%%%%%%%%%%%%%%%%%
	\subsection{Model problem}\label{subsection:model_problem}
	%%%%%%%%%%%%%%%%%%%%%%%%%%%%%%%%%%%%%%%%%%%%%
	For $  {d \! \in \! \{ 1, 2, 3 \}}$, let ${\Omega \! \subset \! \RR^d}$ be a bounded Lipschitz domain with polygonal boundary $\partial \Omega$.
	Given $f \in L^2(\Omega)$ and $\f \in [L^2(\Omega)]^d$, we consider the second-order linear elliptic diffusion problem
	\begin{align} \label{equation:strong_form}
		\begin{split}
			- {\rm div} (\K \nabla u^{\star})
			&= f - {\rm div} \f \phantom{0} \text{in } \Omega,\\
			u^{\star} &= 0 \phantom{f - {\rm div} \f}  \text{on }  \partial \Omega,
		\end{split}
	\end{align}
	where $\K \in [L^\infty(\Omega)]^{d \times d}_{\mathrm{sym}}$ is the symmetric and uniformly positive definite diffusion coefficient. More precisely, given a conforming simplicial triangulation $\TT_\coarse$ of $\Omega$ into compact simplices, we have $\K\vert_T \in [ W^{1,\infty}(T)]^{d \times d}$ for all $T \in \TT_\coarse$. For $x \in \Omega$ we denote the maximal and minimal eigenvalue of $\K(x) \in \R^{d \times d}_{\rm sym}$ by $\lambda_{\max}(\K(x))$ and $\lambda_{\min}(\K(x))$, respectively, and define $\Lambda_{\max} \coloneqq \esssup_{x \in \Omega} \lambda_{\max}(\K(x))$ as well as $\Lambda_{\min} \coloneqq \essinf_{x \in \Omega} \lambda_{\min}(\K(x))$. With $\langle \cdot \, , \, \cdot \rangle_{\omega}$ denoting the usual $L^2(\omega)$-scalar product for a measurable subset $\omega \subseteq \Omega$,
	the weak formulation of~\eqref{equation:strong_form} seeks
	$u^{\star} \in \V := H^1_0(\Omega)$
	solving
	\begin{align}\label{equation:weak_form}
		\edual{u^{\star}}{v}_{\Omega}
		:= \dual{\K \nabla u^\star}{\nabla v}_{\Omega}
		= \dual{f}{v}_{\Omega} + \dual{\f}{\nabla v}_{\Omega}
		\eqqcolon F(v)
		\quad \text{for all } v \in  \V.
	\end{align}
	We note that $\edual\cdot\cdot_{\Omega}$ is a scalar product and the induced semi-norm $\enorm{u}_{\Omega}^2 := \edual{u}{u}_{\Omega}$ is an equivalent norm on $\V$. Therefore, the Lax--Milgram lemma yields existence and uniqueness of the weak solution $u^{\star} \in \V$. For $\omega = \Omega$, we omit the index $\omega$ throughout.
	
	To discretize~\eqref{equation:weak_form}, denote for a polynomial degree $p \ge 1$ and a triangle $T \in \TT_\coarse$ the space of all polynomials on $T$ of degree at most $p$ with $\P^{p}(T)$ and define
	\begin{align}\label{equation:discrete_spaces}
		\S^q(\TT_\coarse) \! := \set{v_\coarse  \in   C (\Omega)}{v_\coarse|_T  \in   \P^{q}(T) \ \text{for all } T  \!\in \! \TT_\coarse} \quad \text{with } q \in \{1,p\}.
	\end{align}
	With the definition $\V_\coarse^p \coloneqq \S^p_0(\TT_\coarse) \coloneqq \S^p(\TT_\coarse) \cap H^1_0(\Omega)$,
	the discrete problem consists of finding $u^{\star}_{\coarse} \in \Vp_{\coarse}$ such that
	\begin{equation}\label{equation:discrete_form}
		\begin{aligned}
			\edual{u^{\star}_{\coarse}}{v_{\coarse}}&= F(v_{\coarse})
			\quad \text{for all } v_{\coarse} \in \V_\coarse^p.
		\end{aligned}
	\end{equation}
	Clearly, the formulation of the discrete problem~\eqref{equation:discrete_form} hinges on the choice of the mesh \(\TT_\coarse\), which directly influences the quality of $u^{\star}_{\coarse}$ as an approximation of $u^{\star}$.
	Note that~\eqref{equation:discrete_form} can be rewritten as a symmetric and positive definite linear system by introducing a basis of $\V_\coarse^p$. However, we opt to work instead with the functional basis-independent description.
	
	%%%%%%%%%%%%%%%%%%%%%%%%%%%%%%%%%%%%%%%%%%%%%
	\subsection{Mesh and space hierarchy}\label{subsection:mesh_hierarchy}
	%%%%%%%%%%%%%%%%%%%%%%%%%%%%%%%%%%%%%%%%%%%%%
	We suppose that the refinement strategy in the module \texttt{REFINE} is newest vertex bisection (NVB); see, e.g., \cite{trax_97, stevenson2008} and Figure~\ref{figure:nvb} for an illustration in \(2\mathrm{D}\).
	\begin{figure}[ht!]
		\includegraphics[width=0.22\textwidth]{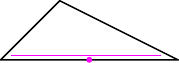} \quad
		\includegraphics[width=0.22\textwidth]{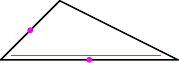} \quad
		\includegraphics[width=0.22\textwidth]{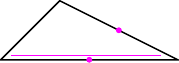} \quad
		\includegraphics[width=0.22\textwidth]{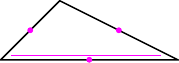} \\
		\includegraphics[width=0.22\textwidth]{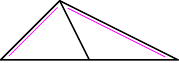} \quad
		\includegraphics[width=0.22\textwidth]{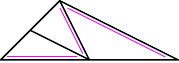}\quad
		\includegraphics[width=0.22\textwidth]{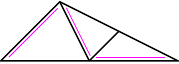}\quad
		\includegraphics[width=0.22\textwidth]{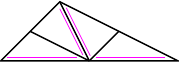}
		\caption{\label{figure:nvb}Schematic of 2D NVB refinement pattern: For each triangle $T \in \TT$, there is one fixed \textsl{refinement edge} $E_T$ indicated by the extra pink line. The pink dots indicate edges that are marked for refinement. If an element is marked for refinement, at least its refinement edge is marked for refinement (top). Iterated bisection refines a marked element into 2, 3, or 4 children (bottom).}
	\end{figure}
	
	Let $\TT_0$ be the conforming initial mesh. We refer to \cite{stevenson2008} for NVB with admissible $\TT_0$ and \(d \ge 2\), to \cite{Karkulik2013a} for NVB with general \(\TT_0\) for \(d=2\), and to the recent work \cite{dgs2023} for NVB with general $\TT_0$ in any dimension. Throughout, we suppose that \(\TT_0\) is admissible. In the $1\mathrm{D}$ case, \cite{Aurada2013a} splits each element into two children of half-length and additionally ensures that any two neighboring elements have uniformly comparable diameter. Let $\T \coloneqq \T(\TT_0)$ be the set of all refinements of $\TT_0$ that can be obtained by arbitrarily many steps of NVB.
	
	From now on, suppose that we are given a sequence $\{ \TT_\ell \}_{\ell = 0}^L \subset \T$ of successively refined triangulations, i.e., for all $\ell = 1, \dots, L$, it holds that $\TT_\ell = {\tt REFINE}(\TT_{\ell-1}, \MM_{\ell-1})$ is the coarsest conforming triangulation obtained by NVB, where all marked elements $\MM_{\ell-1} \subseteq \TT_{\ell-1}$ have been refined by (at least) one bisection.
	We note that NVB-refinement generates meshes that are uniformly $\gamma$-shape regular, i.e., 
	%for all sequences $\{ \TT_\ell \}_{\ell = 0}^L \subset \T$ there holds
	\begin{subequations}\label{equation:shape_regularity}
		\begin{align}\label{equation:gamma_shape}
			\max_{\ell = 0, \dots, L}  \max \limits_{T  \in \TT_\ell} \frac{\diam(T)}{\vert T \vert^{1/d}} \le \gamma < \infty,
		\end{align}
		and
		\begin{equation}\label{equation:shape_regularity2}
			\max_{\ell = 0, \dots, L} \max_{T  \in \TT_{\ell}} \max_{\substack{T^\prime  \in \TT_{\ell} \\ T \cap T^\prime \neq \emptyset}} \frac{\diam(T)}{\diam(T^\prime)} \le \gamma < \infty.
		\end{equation}
	\end{subequations}
	where $\gamma$ depends only on $\TT_0$ and is, in particular, independent of $L$ and the meshes $\TT_1, \dots, \TT_L$; see, e.g., \cite[Theorem~2.1]{stevenson2008} for \(d \ge 2\) or \cite{Aurada2013a} for \(d=1\). 
	We note that~\eqref{equation:gamma_shape} implies \eqref{equation:shape_regularity2} for $d \ge 2$, while \eqref{equation:gamma_shape} is trivial with \(\gamma = 1\) and independent of \eqref{equation:shape_regularity2} for \(d=1\).
	In addition, we define the quasi-uniformity constant
	\begin{align}\label{equation:coarse_quasiuniformity}
		C_{\rm qu} := \min\set{\diam(T)/\diam(T')}{T, T' \in \TT_0} \in (0,1].
	\end{align}
	
	For each mesh $\TT_\ell$, let $\VV_\ell$ denote the set of vertices. Given a vertex $z \in \VV_\ell$, we denote by $\TT_{\ell,z} :=  \set{T \in \TT_\ell}{z \in T}$ the patch of elements of $\TT_\ell$ that share the vertex $z$.
	The corresponding (open) patch subdomain is denoted by $\omega_{\ell, z} := {\rm interior}(\bigcup_{T \in \TT_{\ell,z}} T)$ and its size by $h_{\ell, z} \coloneqq \max_{T  \in \TT_{\ell, z}} h_T \coloneqq \max_{T  \in \TT_{\ell, z}} \vert T \vert^{1/d}$. Finally, we
	denote by $\VV_\ell^+$ the set of new vertices in $\TT_\ell$ and the pre-existing vertices of $\TT_{\ell-1}$ whose associated patches have shrunk in size in the refinement step $\ell$, i.e.,
	\[
	\VV_0^+ := \VV_0
	\quad \text{and} \quad
	\VV_\ell^+ := \VV_\ell \setminus \VV_{\ell-1} \cup \set{z \in \VV_\ell \cap \VV_{\ell-1}}{\omega_{\ell, z}  \neq \omega_{\ell-1,z}}
	\quad \text{for } \ell \ge 1.
	\]
	While this notation is used in the analysis of the solver below, the presentation of Algorithm~\ref{algorithm:solver} is more compact with the abbreviation $\mathcal{N}_\ell = \VV_\ell^+$ for $\ell = 1, \dots, L-1$ and $\mathcal{N}_L \coloneqq \VV_{L}^+$ for $p=1$ and $\mathcal{N}_L \coloneqq \VV_L$ otherwise, where we recall that $p \in \N$ is the fixed polynomial degree of the FEM ansatz functions.
	
	From the definition of the discrete FEM spaces~\eqref{equation:discrete_spaces} and NVB-refinement, we see that there holds nestedness
	\begin{align} \label{equation:space_nestedness}
		\V_0^1 \subseteq \V_1^1 \subseteq \dots \subseteq \V_{L-1}^1 \subseteq \V_L^p.
	\end{align}
	Furthermore, we require the local spaces
	\begin{align} \label{equation:local_spaces}
		\V^q_{\ell,z} := \S^q_0(\TT_{\ell, z})
		\quad \text{for all vertices } z \in \VV_\ell \text{ and } q \in \{1,p\},
	\end{align}
	where we use $q = 1$ for $\ell = 0, \dots, L-1$ and $q = p$ for $\ell = L$; see Figure~\ref{Small_patch} for the illustration of the degrees of freedom for $p =2$.

	\begin{figure}[htpb!]
		\centering
		\includegraphics[scale=0.9]{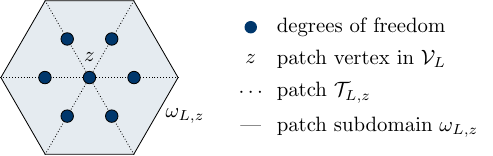}
		\caption{Illustration of degrees of freedom ($p=2$) for the space
			$\V_{L,z}^{p}$ associated to the patch $\TT_{L,z}$.}
		\label{Small_patch}
	\end{figure}

	%%%%%%%%%%%%%%%%%%%%%%%%%%%%%%%%%%%%%%%%%%%%%
	%%%%%%%%%%%%%%%%%%%%%%%%%%%%%%%%%%%%%%%%%%%%%
	\subsection{Multigrid solver}\label{section:multilevel_solver}
	%%%%%%%%%%%%%%%%%%%%%%%%%%%%%%%%%%%%%%%%%%%%%
	%%%%%%%%%%%%%%%%%%%%%%%%%%%%%%%%%%%%%%%%%%%%%
	In the following, we introduce a local geometric multigrid method, which will serve as iterative solver within the \texttt{SOLVE} module of an adaptive FEM algorithm.
	Each full step of the proposed multigrid method can be mathematically described by an iteration operator 	$\Phi \colon \V_L^p \to \V_L^p$, i.e., given the current approximation $u_L \in \V_L^p$, the solver generates the new iterate $\Phi(u_L) \in \V_L^p$.
	
	The main ingredients in the solver construction are an inexpensive global residual solve on $\TT_0$ and local residual solves on all patches $\TT_{\ell, z}$ for $z\in \VV_{\ell}^+$ on the intermediate levels $\ell = 1, \ldots, L-1$ and all patches on the finest level $\TT_L$ when $p>1$.
	For ease of notation, we define the algebraic residual functional $R_L \colon \V_L^p \to \RR$ by
	\begin{align} \label{equation:residual_functional}
		v_L \in \V_L^p \mapsto R_L(v_L) := F(v_L) - \edual{u_L}{v_L} = \edual{u_L^\star - u_L}{v_L} \in \RR.
	\end{align}
	To construct the new iterate $\Phi(u_L)$, levelwise residual liftings of the algebraic error are added to the current approximation $u_L$. The same levelwise residual liftings are used to define an \textsl{a~posteriori} error estimator $\zeta_L(u_L)$ for the algebraic error, i.e., the solver comes with a built-in estimator.
	
	\begin{algorithm}[One step of the optimal local multigrid solver]
		\label{algorithm:solver}
		{\bfseries Input:} Current approximation $u_L \in \V_L^p$, meshes $\{\TT_\ell \}_{\ell = 0}^L$, polynomial degree $p \in \mathbb{N}$.\\
		{\bfseries Solver step:} Perform the following steps~{\rm(i)--(ii)}:
		\begin{enumerate}
			\item[\rm(i)] {\bfseries Global lowest-order residual problem on the coarsest level}:
			\begin{itemize}
				\item Compute $\rho_0 \in \V_0^1$ by solving
				\begin{align}\label{equation:alg_step1}
					\edual{\rho_0}{v_0} = R_L(v_0) \quad \text{for all } v_0 \in \V_0^1.
				\end{align}
				\item Define step-size $\lambda_0 := 1$.
				\item Initialize algebraic lifting $\sigma_0 := \lambda_0 \rho_0$ and \textsl{a~posteriori} estimator ${\zeta_0^2 \! := \! \enorm{\lambda_0 \rho_0}^2}\!$.
			\end{itemize}
			
			\item[\rm(ii)] {\bfseries Local residual-update:}
			For all
			$\ell = 1, \dots, L$, do the following steps, where $q = 1$ for $\ell = 1, \dots, L-1$ and $q = p$ for $\ell = L$:
			
			\begin{itemize}
				\item For all $z \in \mathcal{N}_\ell$, compute $\rho_{\ell, z} \in \V_{\ell,z}^{q}$  by solving
				\begin{align}\label{equation:alg_step2}
					\edual{\rho_{\ell, z}}{v_{\ell, z}}
					= R_L(v_{\ell, z}) - \edual{\sigma_{\ell-1}}{v_{\ell, z}}
					\quad \text{for all } v_{\ell, z} \in \V_{\ell,z}^{q}.
				\end{align}
				
				\item Define  the line-search step-size $s_\ell := (R_L(\rho_\ell) - \edual{\sigma_{\ell-1}}{\rho_\ell})/\enorm{\rho_\ell}^2$, with $
				\rho_\ell := \sum_{z \in \mathcal{N}_\ell} \rho_{\ell, z}$ and the understanding that $0/0 \coloneqq 0$ if $\rho_\ell = 0$, and

				\begin{align*}
					\lambda_\ell  \! \coloneqq
					\begin{cases}
						s_\ell &\text{if } s_\ell \le d+1 \
						\text{ or $\big[\ell = L$ and $p>1\big]$},
						\\
						(d+1)^{-1} & \, \textrm{otherwise.}
					\end{cases}
				\end{align*}

				\item Update $\sigma_\ell := \sigma_{\ell-1} + \lambda_\ell \rho_\ell$ and $\zeta_\ell^2 := \zeta_{\ell-1}^2 + \lambda_\ell \sum_{z \in \mathcal{N}_\ell} \enorm{\rho_{\ell,z}}^2$.
			\end{itemize}
			
		\end{enumerate}
		
		{\bfseries Output:} Improved approximation $\Phi(u_L) := u_L + \sigma_L \in \V_L^p$ and associated \textsl{a~posteriori} estimator $\etalg (u_L) := \zeta_L$ of the algebraic error.
	\end{algorithm}
	
	\begin{remark}[Construction of the new iterate]\label{remark:iterate_construction}
		The construction of $\Phi(u_L)$ from $u_L$ by Algorithm~\ref{algorithm:solver} can be seen as one iteration of a V-cycle multigrid with no pre- and one post-smoothing step, and a step-size at the error correction stage. The smoother on each level is additive Schwarz associated to patch subdomains where the local problems~\eqref{equation:alg_step2} are defined. This is equivalent to diagonal Jacobi smoothing for $p=1$ (e.g., on intermediate levels) and block-Jacobi smoothing for $p>1$ (e.g., on the finest level).
		The choice and use of the step-sizes $\lambda_\ell$ in Algorithm~\ref{algorithm:solver} {\rm(ii)} comes from a line-search approach; see, e.g.,~\cite[Lemma 4.3]{Mir_Pap_Voh_lam_21} and one of the earlier works~\cite{Heinrichs_88_line_relaxation_MG}. However, if the step-size from the line-search is too large, we use instead a fixed damping parameter offsetting the $d+1$ patch overlaps. We note that this case never occurred in practice in any of our numerical experiments.
	\end{remark}
	
	\begin{remark}[Computational effort and speed of convergence]
		We note that we apply a patchwise Cholesky factorization on the finest level. Hence, the computational effort for the local residual solve on the finest mesh $\TT_L$ in dependence on the polynomial degree $p$ is of order $\mathcal{O}(p^{3d}\#\TT_L)$. The presented algorithm is a linear method. One could symmetrize the procedure by adding one pre-smoothing step to define a preconditioner in the hope of accelerating convergence with the help of conjugate gradients. However, in our experience, the patchwise pre-smoothing typically did not yield considerable algebraic error decrease; see, e.g.~\cite{DiPietro_Huelsemann_Matalon_Mycek_Ruede_Ruiz_21_HHO_MG}, while still doubling the number of smoothing operations of a V-cycle. The remaining steps needed to compute the new approximation stem from classical multigrid solvers (such as intergrid operators). We stress that the overall effort does not depend on the number of levels \(L\).
	\end{remark}
	
	\begin{remark}[Nested iterations]
		In the context of adaptive FEM, the 
		solver does not start from an arbitrary initial guess on each newly-refined mesh but from the final approximation of the previous level (see Algorithm~\ref{algorithm:afem} below). This will ensure \textsl{a~posteriori} error control in each step after initialization as well as optimal computational cost. From the algebraic solver perspective, such an approach can be seen as a full multigrid method over the evolving hierarchy of meshes whose number of cycles is determined by the adaptive stopping criterion.
	\end{remark}
	
	%%%%%%%%%%%%%%%%%%%%%%%%
	\subsection{Main result} \label{section:main_theorems}
	%%%%%%%%%%%%%%%%%%%%%%%%
	
	This subsection formulates the main results regarding the iterative solver stating the \emph{contraction} of the multigrid solver and \emph{reliability} of the built-in \textsl{a~posteriori} estimator of the algebraic error. Both results hold \emph{robustly} in the number of levels $L$ and the polynomial degree $p$.
	
	\begin{theorem}\label{theorem:solver}
		Let $u_L^\star \in \V_L^p$ be the (unknown) finite element solution of~\eqref{equation:discrete_form} and let $v_L \in \V_L^p$
		be arbitrary. Let $\Phi(v_L) \in \V_L^p$ and $\etalg(v_L)$ be generated by Algorithm~\ref{algorithm:solver}. Then, the solver iterates and the estimator are connected by
		\begin{align}\label{equation:pythagoras}
			\enorm{u_L^\star - \Phi(v_L)}^2
			\le \enorm{u_L^\star - v_L}^2 -\etalg(v_L)^2.
		\end{align}
		Moreover, the solver contracts the error, i.e., there exists $0 < \qctr < 1$ such that
		\begin{equation}\label{equation:contraction}
			\begin{aligned}
				\enorm{u_L^\star - \Phi(v_L)} \le  \qctr \, \enorm{u_L^\star - v_L}.
			\end{aligned}
		\end{equation}
		Finally, the estimator is a two-sided bound of the algebraic error, i.e., there exists $ \crel > 1$ such that
		\begin{align}\label{equation:equivalence}
			\etalg(v_L)
			\le \enorm{u_L^\star - v_L}
			\le  \crel \, \etalg(v_L).
		\end{align}
		The contraction and reliability constants $\qctr$ and $\crel$ depend only on the space dimension $d$, the $\gamma$-shape regularity~\eqref{equation:shape_regularity},  the quasi-uniformity constant $C_{\rm qu}$ from~\eqref{equation:coarse_quasiuniformity}, $\Lambda_{\max}/ \Lambda_{\min}$, and \(\max_{T  \in \TT_{L}}\Vert \div(\K) \Vert_{L^{\infty}(T)} / \Lambda_{\min}\). In particular, $\qctr$ is independent of the polynomial degree $p$, the number of mesh levels $L$, and the meshes $\TT_1, \dots, \TT_L$.
	\end{theorem}
	
	\begin{corollary}\label{corollary:solver}
		The reliability of the estimator in~\eqref{equation:equivalence} is equivalent to the solver contraction~\eqref{equation:contraction}.
		In particular, this also yields that
		\begin{align} \label{equation:lower_bound}
			\enorm{u_L^\star - \Phi(v_L)} \le \qctr \crel \, \etalg (v_L).
		\end{align}
	\end{corollary}
	
	\begin{remark}
		We note that \eqref{equation:pythagoras} holds with equality whenever the step-size criterion $s_\ell \le d+1$ in Algorithm~\ref{algorithm:solver}(ii) are fulfilled and the construction is thus done by optimal-line search. In such a case, which was always satisfied in all our numerical tests, a Pythagoras identity in the spirit of~\cite[Theorem~4.7]{Mir_Pap_Voh_lam_21} yielding exact algebraic error decrease is obtained.
	\end{remark}
	
	%%%%%%%%%%%%%%%%%%%%%%%%%%%%%%%%%%%%%%%%%%%%%
	%%%%%%%%%%%%%%%%%%%%%%%%%%%%%%%%%%%%%%%%%%%%%
	\section{Application to adaptive FEM with inexact solver}
	\label{subsection:AFEM}
	%%%%%%%%%%%%%%%%%%%%%%%%%%%%%%%%%%%%%%%%%%%%%
	%%%%%%%%%%%%%%%%%%%%%%%%%%%%%%%%%%%%%%%%%%%%%
	
	%%%%%%%%%%%%%%%%%%%%%%%%%%%%%%%%%%%%%%%%%%%%%%%%%
	%\subsection{Adaptive finite element algorithm}
	%%%%%%%%%%%%%%%%%%%%%%%%%%%%%%%%%%%%%%%%%%%%%%%%%
	
	%Clearly, the formulation of the discrete problem~\eqref{equation:discrete_form} hinges on the choice of the mesh, which directly influences the quality of $u^{\star}_{\coarse}$ as an approximation of $u^{\star}$. 
	Given a coarse mesh $\TT_0$, we
	%Thus, one can start by introducing a coarse mesh and 
	use an adaptive finite element method (AFEM) to generate locally refined meshes $\{\TT_L\}_{L \in \N}$ tailored to the behavior of the sought solution.
	In the spirit of~\cite{GHPS_21}, Algorithm~\ref{algorithm:afem} presents such an approach with an adaptively stopped iterative solver,
	where Step~(Ii) exploits the built-in \textsl{a posteriori} estimator of the geometric multigrid solver from Section~\ref{section:Setting}.
	%{\color{cyan}From this point onward, we focus on Step~(Ii) of Algorithm~\ref{algorithm:afem} for a given adaptive loop counter $L>0$ and $k \ge 1$.}
	
	%The focus of this work is the design of the iterative solver as well as its optimal integration into the standard AFEM loop.
	While we note that the present Algorithm~\ref{algorithm:afem} and the corresponding Theorem~\ref{theorem:optimal-cost} are restricted to fixed polynomial degree $p$, the inclusion of the proposed $hp$-robust iterative solver into the $hp$-adaptive FEM algorithm of~\cite{cnsv2017} remains for future research, since the mathematical understanding of $hp$-adaptive FEM is still widely open.
	
	\begin{algorithm}[AFEM with iterative solver]
		~\label{algorithm:afem}
		\newline
		{\bfseries Input:} Initial mesh $ \TT_0$, polynomial degree $p \in \N$, adaptivity parameters $0 < \theta \le 1$, $C_{\rm mark} \ge 1$, and $\mu > 0$, initial guess $u_0^0 := 0$.
		
		\noindent
		{\bfseries Adaptive loop:} repeat the following steps {\rm(I)--(III)} for all $L = 0, 1, 2, \dots$:
		\begin{enumerate}
			\item[\rm(I)] {\tt SOLVE \& ESTIMATE:} repeat the following steps {\rm(i)--(iii)} for all $k = 1, 2, 3,\dots$:
			\begin{enumerate}
				
				\item[\rm(i)] Do one step of the algebraic solver to obtain $u_L^k  \in \V_L^p$ and an associated \textsl{a~posteriori} estimator $\etalg(u_L^{k-1})$ for the algebraic error
				\begin{align*}
					[u_L^k, \etalg(u_L^{k-1})] \ := \ {\tt SOLVE}(u_L^{k-1}, \{\TT_\ell \}_{\ell = 0}^L,p).
				\end{align*}
				
				\item[\rm(ii)] Compute \textsl{a~posteriori} indicators for the elementwise discretization error
				\begin{align*}
					\{ \etdisc (T,u_L^k) \}_{T \in \TT_L} :=  {\tt ESTIMATE}(u_L^k, \TT_L).
				\end{align*}
				
				\item[\rm(iii)] If $\etalg(u_L^{k-1}) \le \mu \etdisc(u_L^k)$, terminate  the $k$-loop, set the index \(\kk[L] \coloneqq k\) and define $u_L := u_L^{\kk[L]}$.
				
			\end{enumerate}
			
			\item[\rm(II)] {\tt MARK:} Determine a set of marked elements $\MM_L \subseteq \TT_L$ of (up to the multiplicative constant $C_{\rm mark}$) minimal cardinality that satisfies
			\begin{align*}
				\theta \, \etdisc (u_L)^2 \le
				\sum_{T \in \MM_L} \etdisc (T,u_L)^2.
			\end{align*}
			\item[\rm(III)]  {\tt REFINE:} Generate the new mesh $\TT_{L+1} :=  {\tt REFINE} (\MM_L, \TT_L)$ and define $u_{L+1}^0 := u_L $.
		\end{enumerate}
		
		\noindent
		{\bfseries Output:}  Sequences of successively refined triangulations $\TT_L$, discrete approximations $u_L$ and corresponding error estimators $(\eta_L(u_L),\etalg(u_L))$.
	\end{algorithm}
	Mesh-refinement is steered by the discretization error estimator. For all $T \in \TT_\coarse$, let $\eta_\coarse(T; \cdot)\colon \V_\coarse^p \to \R_{\ge 0}$
	%  $\eta_\coarse$: let
	% \begin{align}
		% 	\begin{split}\label{eq:estimator:generic}
			% 		\eta_\coarse(T; \cdot)\colon \V_\coarse^p \to \R_{\ge 0}\quad \text{ for all } T \in \TT_\coarse
			% 	\end{split}
		% \end{align}
	be the local contributions of the standard residual error estimator defined by
	\begin{align}\label{eq:residual-estimator}
		\eta_\coarse^2(T; v_\coarse) \coloneqq h_T^2 \Vert f + \div(\K \nabla v_\coarse - \f) \Vert_{T}^2 + h_T \Vert \jump{\K \nabla v_\coarse - \f} \cdot \boldsymbol{n} \Vert_{\partial T \cap \Omega}^{2},
	\end{align}
	where $\Vert \cdot \Vert_\omega$ denote the appropriate $L^2(\omega)$-norms.
	We define
	\begin{equation*}
		\eta_\coarse(\UU_\coarse; v_\coarse)
		:=
		\Bigl( \sum_{T \in \,\UU_\coarse} \eta_\coarse(T; v_\coarse)^2 \Bigr)^{1/2}
		\quad
		\text{for all } \UU_\coarse \subseteq \TT_\coarse \text{ and } v_\coarse \in \V_\coarse^{p}.
	\end{equation*}
	To abbreviate notation, let $\eta_\coarse(v_\coarse) := \eta_\coarse(\TT_\coarse; v_\coarse)$.
	
	%%%%%%%%%%%%%%%%%%%%%%%%%%%%%%%%%%%%%%%%%%%%%%%%%
	%\subsection{Adaptive FEM ist quasi-optimal computational cost} 
	%%%%%%%%%%%%%%%%%%%%%%%%%%%%%%%%%%%%%%%%%%%%%%%%%
	
	One important consequence of Theorem~\ref{theorem:solver} is optimal convergence of Algorithm~\ref{algorithm:afem} with respect to computational complexity. To formulate this mathematically, we define the ordered set
	\begin{align*}
		\mathcal{Q} \coloneqq \set{(L, k) \in \N_0^2}{(L, k) \text{ is used as loop variable in Algorithm~\ref{algorithm:afem} and \(1 \le k \le \kk[L]\)}}.
	\end{align*}
	On $\QQ$, we define the ordering $\le$ by
	\begin{align*}
		(L', k') \le (L, k)
		\quad \Longleftrightarrow \quad
		u_{L'}^{k'} \text{ is computed earlier than or equal to $u_L^{k}$ in Algorithm~\ref{algorithm:afem}.}
	\end{align*}
	Furthermore, we introduce the total step counter $\vert \cdot, \cdot\vert$, defined for all $(L, k) \in \QQ$, by
	\begin{align*}
		\vert L, k \vert \coloneqq \# \set{(L', k') \in \QQ}{(L', k') \le (L, k)}.
	\end{align*}
	Before we state the theorem, we introduce the notion of approximation classes. For $\TT \in \T $ and $s>0$ define
	\begin{align}\label{eq:def_approx_class}
		\norm{u}{\A_s} := \sup_{N \in \N_0} \Big( \big( N+1 \big)^s \min_{\TT_{\rm opt} \in \T_N (\TT_0) } \big( \enorm{u^\star - u^\star_{\rm opt}} + \eta_{\rm opt} (u^\star_{\rm opt}) \big) \Big),
	\end{align}
	with Galerkin solution \(u_{\rm opt}^\exact\) and  estimator $\eta_{\rm opt}$ on the optimal triangulation $\TT_{\rm opt} \in \T_N (\TT_0)$, where $\T_N (\TT_0) \coloneqq \set{\TT_\coarse \in \T(\TT_0)}{\#\TT_\coarse -  \#\TT_0 \le N}$. By reliability (A3) of the estimator, see, e.g.,  \cite{Cars_Feis_Page_Prae_ax_adpt_14}, the sum on the right-hand side of \eqref{eq:def_approx_class} is equivalent to $\eta_{\rm opt} (u^\star_{\rm opt})$. If $\norm{u}{\A_s} < \infty$, then we say that rate $s$ is possible.
	%It is well known that the standard residual error estimator satisfies the axioms of adaptivity from \cite{Cars_Feis_Page_Prae_ax_adpt_14} 
	%{\color{cyan} and thus there holds full linear convergence of the quasi-error $\Delta_L^k \coloneqq \enorm{u_L^\exact - u_L^k} + \etdisc(u_L^k)$.}
	In \cite{GHPS_21}, it is shown that in the case of a contractive solver, convergence rates with respect to degrees of freedom are equivalent to convergence rates with respect to computational complexity. We abbreviate with $\cost(L,k)$ the total costs of Algorithm~\ref{algorithm:afem} defined by
	\begin{align*}
		\cost(L,k) \coloneqq \sum_{\substack{(L',k') \in \QQ \\ (L',k') \le (L,k)}} \#\TT_{L'}.
	\end{align*}
	
	\begin{theorem}\label{theorem:optimal-cost}
		Let $\{\TT_L\}_{L \in \N_0}$ be the sequence generated by Algorithm~\ref{algorithm:afem} and define the quasi-error by
			\begin{equation*}
				\Delta_L^k \coloneqq \enorm{u^\exact - u_L^k} + \eta_L(u_L^k) \quad \text{for all } (L, k) \in \QQ.
		\end{equation*} 
		Then, for all parameters $0 < \theta \le 1$ and $\mu > 0$, it holds that
		\begin{equation}\label{eq:rate-equivalence}
			\sup \limits_{(L,k) \in \QQ} (\# \TT_L)^s \Delta_L^k \simeq \sup_{(L,k) \in \QQ}
			\cost(L,k)^{s} \, \Delta_L^{k} \quad \text{ and } \quad \Delta_L^k \to 0 \text{ as } |L, k| \to \infty.
		\end{equation}
		Furthermore, there exist $0 < \theta^\star \le 1$, and $\mu^\star > 0$ such that, for sufficiently small parameters $0 < \theta < \theta^\star$ and $0 < \mu/ \theta < \mu^\star$, and for all $s > 0$, it holds that
		\begin{align}\label{eq:theorem:aisfem:complexity}
			\copt \,\Vert u\Vert_{{\mathbb{A}_s}}
			\le
			\sup_{(L,k) \in \QQ} \cost(L,k)^{s} \, \Delta_L^{k}
			\le \Copt \, \max\{\Vert u \Vert_{{\mathbb{A}_s}}, \Delta_{0}^{0}\}.
		\end{align}
		The constants $\copt, \Copt > 0$ depend only on the polynomial degree $p$, the initial triangulation $\TT_0$, $\Lambda_{\max}/\Lambda_{\min}$, \(\max_{T  \in \TT_{L}}\Vert \div(\K) \Vert_{L^{\infty}(T)} / \Lambda_{\min}\), the rate $s$, the ratios $\theta / \theta^\exact$ and $\mu / (\theta \mu^\exact)$, and the properties of newest vertex bisection.
		In particular, this proves the equivalence
		\begin{align}\label{eq2:theorem:aisfem:complexity}
			\Vert u \Vert_{{\mathbb{A}_s}} < \infty
			\quad \Longleftrightarrow \quad
			\sup_{(L,k) \in \QQ}
			\cost(L,k)^{s} \, \Delta_L^{k} < \infty,
		\end{align}
		which proves optimal complexity of Algorithm~\ref{algorithm:afem}.
	\end{theorem}
	
	\begin{remark}
			We note that in \cite[Theorem~8]{GHPS_21}, the constant \(\copt > 0\) additionally depends on the stopping index \(\underline{L}\) in the case the algorithm terminates after a finite number of mesh levels \(\underline{L} < \infty\) or the estimator satisfies \(\eta_{L}(u_{L}) = 0\). The refined analysis in the recent work \cite{bhimps2023} removes this dependence.
	\end{remark}
	
	\begin{remark}
			We note that it is also possible to use the same stopping criterion for the algebraic solver as in \cite[Algorithm~2]{GHPS_21}. However, since the multigrid solver from Algorithm~\ref{algorithm:solver} has a built-in estimator for the algebraic error, we opt for its choice within Algorithm~\ref{algorithm:afem} instead. 
	\end{remark}
	
	\begin{proof}[Proof of Theorem~\ref{theorem:optimal-cost}]
		We show that Algorithm~\ref{algorithm:afem} satisfies the requirements of~\cite[Theorem~4 and Theorem~8]{GHPS_21}.
		First note that the standard residual error estimator from~\eqref{eq:residual-estimator} satisfies the axioms of adaptivity from \cite{Cars_Feis_Page_Prae_ax_adpt_14} and thus satisfies the assumptions (A1)--(A4) from~\cite[Theorem~8]{GHPS_21}.
		Furthermore, newest vertex bisection satisfies the assumptions (R1)--(R3) from~\cite[Section~2.2]{GHPS_21}. For the present setting, the conditions (C1) and (C2) from \cite[Section~2.5]{GHPS_21} coincide and are satisfied.
		
		Tracing the role of the stopping criterion for the case (C1) in the proof of \cite[Theorem~\revision{4}]{GHPS_21}, one sees that the stopping criterion needs to guarantee that, for all $(L,k) \in \QQ$,
		\begin{equation}\label{eq:stopping-requirement}
			\begin{split}
				\enorm{u_L^k - u_L^{k-1}}
				&\leq \lambda_1 \, \eta_L(u_L^k)
				\qquad \text{if } u_L^k = u_L,\\
				\eta_L(u_L^k)
				&\leq \lambda_2^{-1} \, \enorm{u_L^k - u_L^{k-1}}
				\qquad \text{else,}
			\end{split}
		\end{equation}
		for some $\lambda_1, \lambda_2>0$.
		The upper bound in~\eqref{equation:equivalence} in Theorem~\ref{theorem:solver} as well as contraction~\eqref{equation:contraction} show that, for all $(L,k) \in \QQ$, our stopping criterion in Algorithm~\ref{algorithm:afem}~\textrm{Step~(Iiii)} leads for \(u_L^k = u_L\) to
		\begin{equation*}
				\enorm{u_L^k - u_L^{k-1}} \eqreff{equation:contraction}\le (1+\qctr) \, \enorm{u_L^\exact - u_L^{k-1}} \eqreff{equation:equivalence}\le \crel \, (1+\qctr) \, \zeta_L(u_L^{k-1}) \le \mu \, \crel \, (1+\qctr) \, \eta_L(u_L^k).
		\end{equation*}
			For the remaining case, the contraction~\eqref{equation:contraction} leads to
			\begin{equation*}
				\enorm{u_L^\exact - u_L^k} \eqreff{equation:contraction}\le \qctr \, \enorm{u_L^\exact - u_L^{k-1}} \le 
				\qctr \, \enorm{u_L^\exact - u_L^k} + \qctr \, \enorm{u_L^k - u_L^{k-1}}.
			\end{equation*}
			This implies
			\begin{equation}\label{eq:error_control}
				\enorm{u_L^\exact - u_L^k} \le \frac{\qctr}{1-\qctr} \enorm{u_L^k - u_L^{k-1}}.
			\end{equation}
			The not met stopping criterion in Algorithm~\ref{algorithm:afem}\textrm{(Iiii)}, the lower bound in~\eqref{equation:equivalence}, and \eqref{eq:error_control} show
			\begin{align*}
				\eta_L(u_L^k) < \mu^{-1} \, \zeta_L(u_L^{k-1}) \eqreff{equation:equivalence}\le 
				\mu^{-1} \, \enorm{u_L^\exact - u_L^{k-1}} 
				&\le 
				\mu^{-1} \,  \bigl(\enorm{u_L^\exact - u_L^{k}} + \enorm{u_L^k - u_L^{k-1}}\bigr)
				\\
				&\eqreff*{eq:error_control}\le 
				\mu^{-1} \, \bigl(1+\frac{\qctr}{1-\qctr}\bigr) \,  \enorm{u_L^k - u_L^{k-1}}
				\\
				&= \mu^{-1} \, \bigl(\frac{1}{1-\qctr}\bigr) \,  \enorm{u_L^k - u_L^{k-1}}.
		\end{align*}
		Overall, \eqref{eq:stopping-requirement} is satisfied with 
			\[
			\lambda_1 = \Crel \, (1+\qctr) \, \mu \quad \text{and} \quad
			\lambda_2 = (1-\qctr) \, \mu,
			\]
			and \cite[Theorem~4]{GHPS_21} proves full linear convergence, so that, in particular, \eqref{eq:rate-equivalence} is fulfilled (see the proof of \cite[Theorem~8]{GHPS_21} or \cite[Corollary~4.2]{bhimps2023}).
		
		The lower bound in \eqref{eq:theorem:aisfem:complexity} follows as in \cite[Theorem~8]{GHPS_21} or \cite[Theorem~4.3]{bhimps2023}. For the upper bound in \eqref{eq:theorem:aisfem:complexity}, \cite[Theorem~8]{GHPS_21} requires that \[
			0 < \lambda_1 / \theta <  \lambda_{\mathrm{opt}} \coloneqq (1-\qctr) / (\qctr \, \Cstab)
			\] 
			and 
			\[
			0 < \theta^\prime \coloneqq \frac{\theta + \lambda_1 / \lambda_\mathrm{opt}}{1 - \lambda_1 / \lambda_{\mathrm{opt}}} < \theta_{\mathrm{opt}} \coloneqq (
			1+ \Cstab^2 \, \Cdrel^2)^{-1},
			\]
			where \(\Cstab\) is the stability constant from (A1) and \(\Cdrel\) is the constant from discrete reliability (A4); see, e.g., \cite{GHPS_21}. We define 
			\[
			\mu^\exact \coloneqq \frac{\lambda_{\mathrm{opt}}}{\Crel \,(1+\qctr)},\]
			and \(\mu / \theta < \mu^\exact\) thus implies \(\lambda_1 / \theta = \Crel \,(1+\qctr) \, \mu / \theta < \lambda_\mathrm{opt} \). Finally, we choose \(\theta^\exact\) such that any \(0 < \widetilde{\theta} \le \theta^\exact\) satisfies \(\frac{2 \, \widetilde{\theta}}{1 - \widetilde{\theta}} < \theta_{\mathrm{opt}}\). Then, \(0 < \theta < \theta^\exact\) yields 
			\(
			\theta^\prime = \frac{\theta + \lambda_1 / \lambda_\mathrm{opt}}{1 - \lambda_1 / \lambda_{\mathrm{opt}}} < \frac{2 \, \theta}{1-\theta} < \theta_{\mathrm{opt}}
			\)
			and
		optimal cost in Theorem~\ref{theorem:optimal-cost} follows directly from~\cite[Theorem~8]{GHPS_21}.
	\end{proof}
	
	%%%%%%%%%%%%%%%%%%%%%%%%%%%%%%%%%%%%%%%%%%%%%
	%%%%%%%%%%%%%%%%%%%%%%%%%%%%%%%%%%%%%%%%%%%%%
	\section{Numerical experiments}\label{section:numerical_experiments}
	%%%%%%%%%%%%%%%%%%%%%%%%%%%%%%%%%%%%%%%%%%%%%
	%%%%%%%%%%%%%%%%%%%%%%%%%%%%%%%%%%%%%%%%%%%%%
	This section investigates the numerical performance of the proposed multigrid solver of Algorithm~\ref{algorithm:solver} and the adaptive Algorithm~\ref{algorithm:afem}. The {\sc Matlab} implementation of the multigrid solver is embedded into the MooAFEM\footnote{available under \url{https://www.asc.tuwien.ac.at/praetorius/mooafem}.} framework from \cite{MooAFEM}. Throughout, we choose the marking parameter $\theta = 0.5$ in the adaptive Algorithm~\ref{algorithm:afem} and $\f = (0, 0)^\top$. We introduce the following test case:

	\begin{itemize}
		\item \emph{L-shaped domain.} Let $\Omega = (-1,1)^2 \setminus \big([0,1] \times [-1,0]\big)$ with right-hand side $f = 1$ and $\K = \boldsymbol{I}$.
	\end{itemize}
	
	%%%%%%%%%%%%%%%%%%%%%%%%%%%%%%%%%%%%%%%%%%%%%
	\subsection{Contraction and performance of local multigrid solver}
	%%%%%%%%%%%%%%%%%%%%%%%%%%%%%%%%%%%%%%%%%%%%%
	
	We confirm numerically our main results from Theorem~\ref{theorem:solver}. In order to study the algebraic solver and its built-in estimator with respect to different polynomial degrees,
	we take $\mu = 10^{-5}$ in Algorithm~\ref{algorithm:afem}, thus \emph{oversolving} the algebraic problem. Moreover, we stop the adaptive algorithm once the final mesh consists of \(10^{6}\) degrees of freedom. 
	%We focus on this last hierarchy with $L$ fixed and investigate the contraction factors of the solver per iteration.
	Note that thanks to
	Corollary~\ref{corollary:solver} proving the equivalence of the reliability of the algebraic error estimator with the contraction of the algebraic solver, we indeed only need to investigate numerically the existence of the $p$-robust bound on the contraction of the solver.
	In Figure~\ref{figure:lshape_contraction} (left), we present the maximal contraction factors on each level \(L\) of the adaptive algorithm from Algorithm~\ref{algorithm:afem}. We see that the contraction factors are robust in the polynomial degree \(p\) with an upper bound of about \(0.7\) in all our experiments. In Figure~\ref{figure:lshape_contraction} (right),  we see that on a fixed number of levels (\(L\)=10) even for higher-order polynomials their behavior is clustered around similar values.  Moreover, from a purely solver-centric perspective, we see that the solver variant which employs higher-order smoothing also on the intermediate levels
	(and not only on the finest one)  as studied in~\cite{Mir_Pap_Voh_lam_21} only leads to slight improvements of the contraction constants. 
	Adapting the arguments of~\cite{Mir_Pap_Voh_lam_21}, this modified construction can be guaranteed to be contractive with $p$-robust, but linearly $L$-dependent contraction bound on the algebraic error.
	%For this modified construction, one would by following similar proofs of~\cite{Mir_Pap_Voh_lam_21} have a $p$-robust, but linearly $L$-dependent, theoretical contraction bound on the algebraic error for one solver step.
	However, this degradation with increasing $L$ is not seen in practice, provided that the patchwise smoothing is done \emph{everywhere} for
	\emph{level $L=1$} (as new degrees of freedom are added on all patches when the polynomial degree is $p>1$) and \emph{local} patchwise smoothing is employed in the remaining levels.
	We present a comparison of the resulting contraction factors of this approach to Algorithm~\ref{algorithm:solver} for a fixed number of level ($L=10$) in Figure~\ref{figure:lshape_contraction}(right).
	%as well as add the results of the $L=25$ case for comparison to Figure~\ref{figure:lshape_contraction} (left).
	
	\begin{figure}[htp!]
		\centering
		\subfloat{
			\includegraphics[scale=0.85]{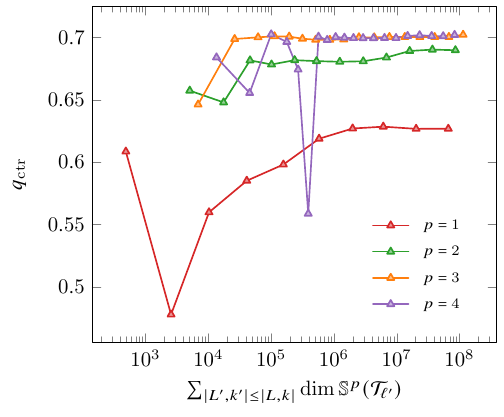}
		}
		\subfloat{
			\includegraphics[scale=0.85]{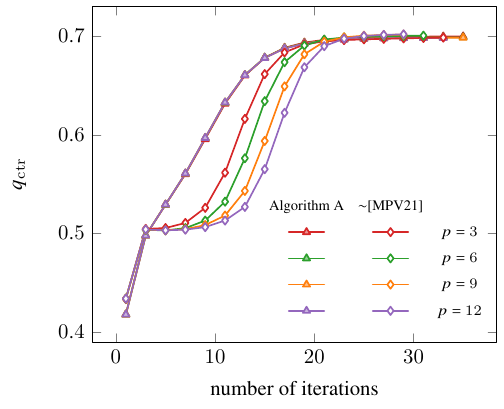}
			%	\label{figure:lshape_error}
		}
		\caption{\label{figure:lshape_contraction}\emph{Contraction of the algebraic solver.} History plot of the contraction factors $\qctr$ from~\eqref{equation:contraction} for various polynomial degrees $p$ with parameter \(\mu = 10^{-5}\) for the presented polynomial hierarchy from \eqref{equation:space_nestedness} in the adaptive algorithm from Algorithm~\ref{algorithm:afem} stopping once the final mesh consists of \(10^6\) degrees of freedom (left) and the comparison with polynomial hierarchy motivated by~\cite{Mir_Pap_Voh_lam_21} with localized smoothing for a fixed number of levels $L=10$ (right).}
	\end{figure}
	
	%%%%%%%%%%%%%%%%%%%%%%%%%%%%%%%%%%%%%%%%%%%%%
	\subsection{Optimality of the adaptive algorithm}
	%%%%%%%%%%%%%%%%%%%%%%%%%%%%%%%%%%%%%%%%%%%%%
	
	We take $\mu = 0.1$ in Algorithm~\ref{algorithm:afem} and study the decrease of the discretization error estimator $\eta_L(u_L)$, both in terms of number of degrees of freedom and timing. We remark that the error estimator $\eta_L(u_L)$ on the final iterates is equivalent to the quasi-error $\Delta_L$. After a pre-asymptotic phase, we see in Figure~\ref{figure:lshape_estimator} and \ref{figure:lshape_cost} for different polynomial degrees $p$ that the optimal convergence rate $-p/2$ is recovered both with respect to number of degrees of freedom and computational time, and the singularity at the reentrant corner $(0,0)$ is resolved through local mesh refinement. Furthermore, Figure~\ref{figure:lshape_timing} shows that the proposed multigrid solver behaves faster than the built-in direct solver ({\sc Matlab} backslash operator) concerning the time per dof. 
	The displayed timings include the setup of the linear system, the time for the solver module, computation of estimator, and mesh refinement. 
	Overall, the numerical experiments in Figure~\ref{figure:lshape_timing} validate the linear complexity of the suggested local multigrid solver from Algorithm~\ref{algorithm:solver}.
	
	\begin{figure}[htp!]
		\subfloat{
			\label{figure:lshape_estimator}
			\includegraphics[scale=0.85]{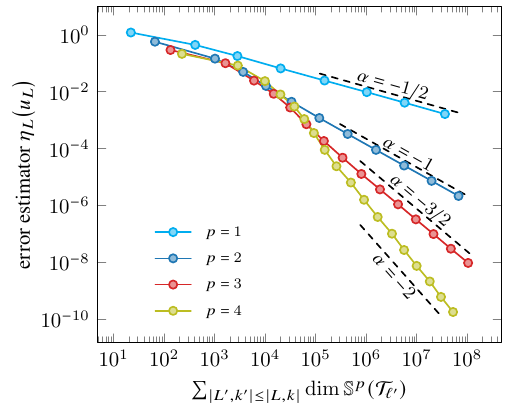}
		} \quad
		\subfloat{
			\label{figure:lshape_cost}
			\includegraphics[scale=0.85]{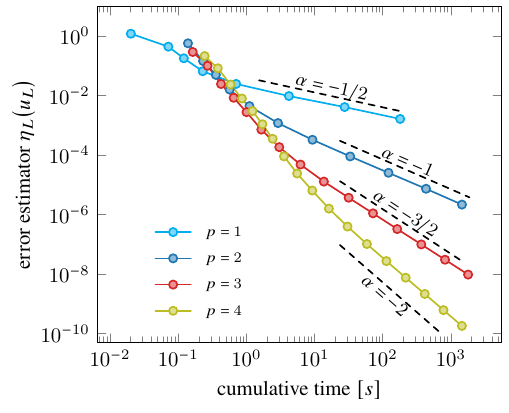}
		} \hfill
		\caption{\emph{Optimality of AFEM on L-shape.}\label{fig:optimality_afem} The convergence history plot of the discretization error estimator $\etdisc(u_L)$ with respect to the total computational cost (left) and the cumulative computational time (right).}
	\end{figure}
	
	\begin{figure}[htbp!]
		\centering
		\subfloat{\includegraphics[scale=0.85]{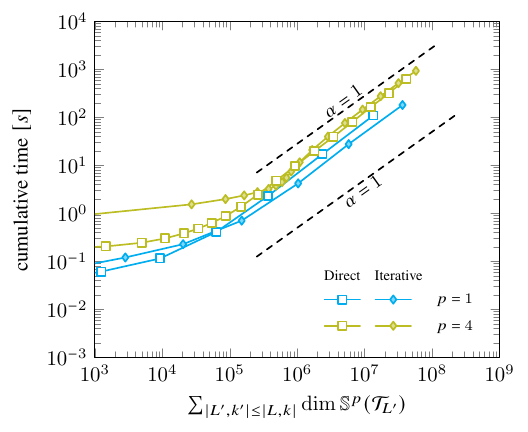}} \quad
		\subfloat{\includegraphics[scale=0.85]{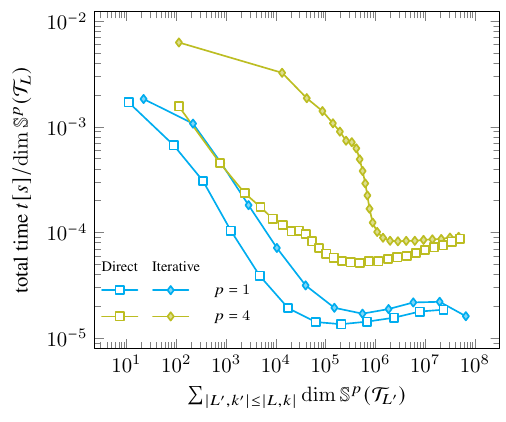}}
		
		\caption{\label{figure:lshape_timing}\emph{Optimality of the local multigrid solver.} History plot of the cumulative computational time and the relative computational time per degree of freedom for the polynomial degrees $p=1$ and $p=4$. We compare the overall time with the direct solve (square) to the overall time of the AFEM algorithm with the multigrid solver (diamond). In particular, the displayed times include setup, marking, and mesh refinement.}
	\end{figure}
	
	%%%%%%%%%%%%%%%%%%%%%%%%%%%%%%%%%%%%%%%%%%%%%
	\subsection{Numerical performance and insights for jumping coefficients}
	%%%%%%%%%%%%%%%%%%%%%%%%%%%%%%%%%%%%%%%%%%%%%
	
	We consider two additional test cases with jumps in the diffusion coefficient:
	
	\begin{itemize}
		\item \emph{Checkerboard.}
		Let $\Omega = (0,1)^2$ be the unit square and $\K$ the $2 \times 2$ checkerboard diffusion with values $1$ (white) and $10^k$ (grey) for fixed $k=1,2,3$, see Figure~\ref{fig:meshes} (left).
		\item \emph{Striped diffusion.} Let $\Omega = (0,1)^2$ be the unit square split into \(2^k\) stripes for \(k=1,2,3\). The value of $\K$ on the \(j\)-th stripe is \(10^{j-1}\) with \(j \in \{1, \ldots, 2^k\}\), see Figure~\ref{fig:meshes}(right).
	\end{itemize}
	
	In Table~\ref{table:convergence}, we see the optimal convergence of the discretization estimator with the optimal rate $-1/2$ for $p=1$ as well as $-1$ for $p=2$ for both diffusion coefficients regardless of the jump size. We stress that the discontinuity in the diffusion coefficient does not affect the optimality of the proposed adaptive algorithm and the iteration numbers remain uniformly bounded as displayed in Table~\ref{table:iterations}.

		Both test cases exhibit singularities due to jumps in the diffusion coefficient; however, the jump can be much higher for two neighboring elements in the checkerboard case. In this case, near the cross point $(1/2, 1/2)$, the jump is of order $10^k$ from one element to the next, which coincides with the jump from the highest to the lowest value of $\K$ on the \emph{whole} domain. For the striped test case, the jump between two neighboring elements belonging to different ``stripes'' is of order $10$, even if the \emph{global} jump in the diffusion (for non-neighboring elements) is of order $10^{2^k-1}$.
	
	%In the Checkerboard case, there is a cross point at the point $(1/2, 1/2)$, where the singularity is strongest. For any two neighboring elements, the jump in the diffusion coefficient is at most of order $10^k$, which coincides with the jump from the highest to the lowest value of $\K$ on the \emph{whole} domain. For the striped diffusion case, however, this is not the case: the jump from the highest to the lowest value of $\K$ on the whole domain is not shared by neighboring elements. 
	This gives us the tools to observe numerically if the performance of our method only depends on \emph{local} jumps in the diffusion coefficient.
	
	\begin{figure}
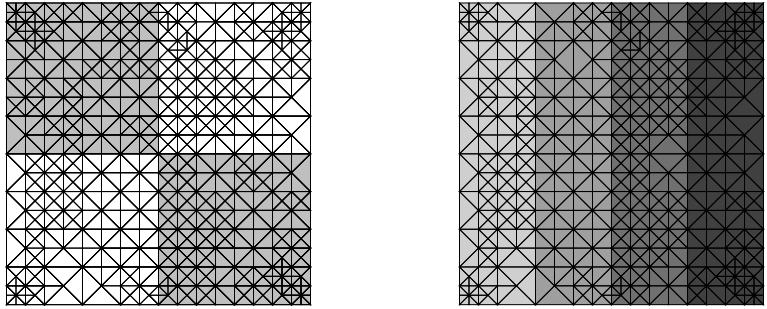

		\centering
		\subfloat{\begin{tikzpicture}[scale=4]
				\input{Numerical_experiments/meshes/Checkerboard_mesh.tex}
				\label{figure:checkerboard_mesh}
		\end{tikzpicture}}
		\hspace{10ex}
		\subfloat{\begin{tikzpicture}[scale=4]
				\input{Numerical_experiments/meshes/Stripe_mesh.tex}
				\label{figure:stripe_mesh}
		\end{tikzpicture}}
		\caption{\label{fig:meshes} Adaptively-refined meshes. Left: checkerboard diffusion with $k=1$, polynomial degree $p=1$ and $\# \TT_{8} = 603$. Right: stripe diffusion with $k=2$, $p=1$ and $\# \TT_{8} = 753$ (right).}
	\end{figure}
	
	\begin{figure}[htp!]
		\subfloat{
			\label{figure:checkerboard_estimator}
			\includegraphics[scale=0.85]{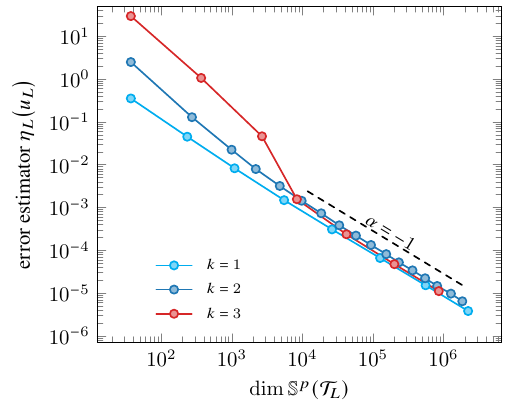}
		}
		\subfloat{
			\label{figure:stripe_estimator}
			\includegraphics[scale=0.85]{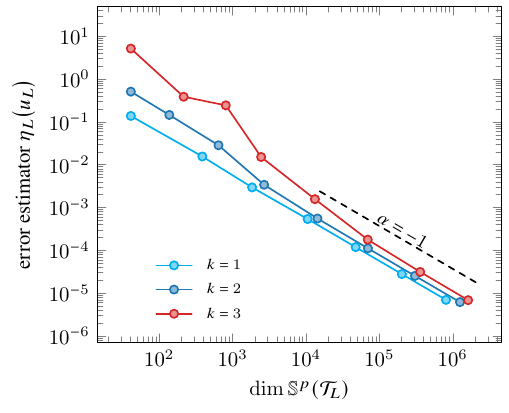}
		} \hfill
		\caption{\emph{Optimality of AFEM for jumping diffusion.} The convergence history plot of the discretization error estimator $\etdisc(u_L)$ for polynomial degree $p=2$ with respect to the total computational cost for the checkerboard diffusion (left) and the stripe diffusion (right).}
	\end{figure}
	
	\begin{table}[htp!]
		\centering
		\begin{tabular}{c|cc|cc|}
			\cline{2-5}
			& \multicolumn{2}{c|}{Checkerboard}  & \multicolumn{2}{c|}{Stripe}      \\ \cline{2-5}
			& \multicolumn{1}{c|}{$p=1$} & $p=2$ & \multicolumn{1}{c|}{$p=1$} & $p=2$ \\ \hline
			\multicolumn{1}{|c|}{$k=1$} & \multicolumn{1}{c|}{-0.4961}      &  -0.9877     & \multicolumn{1}{c|}{-0.4956}      &  -1.0116   \\ \hline
			\multicolumn{1}{|c|}{$k=2$} & \multicolumn{1}{c|}{-0.4960}      &   -0.9946    & \multicolumn{1}{c|}{-0.4969}      &  -0.9670   \\ \hline
			\multicolumn{1}{|c|}{$k=3$} & \multicolumn{1}{c|}{-0.4960}      &   -0.9826    & \multicolumn{1}{c|}{-0.5095}      &  -0.9766   \\ \hline
		\end{tabular}
		\caption{\label{table:convergence} Mean value of experimental convergence rates of the discretization error estimator $\eta_L(u_L)$ over the cumulative cost in a $\log\log$-plot for polynomial degrees $p=1,2$ and diffusion coefficient numbers $k=1, 2, 3$.}
	\end{table}
	
	\begin{table}[htp!]
		\centering
		\begin{tabular}{c|cc|cc|}
			\cline{2-5}
			& \multicolumn{2}{c|}{checkerboard}  & \multicolumn{2}{c|}{stripe}      \\ \cline{2-5}
			& \multicolumn{1}{c|}{$p=1$} & $p=2$ & \multicolumn{1}{c|}{$p=1$} & $p=2$ \\ \hline
			\multicolumn{1}{|c|}{$k=1$} & \multicolumn{1}{c|}{1}      &   1.0455 (mean),  2 (max)     & \multicolumn{1}{c|}{1}      &1.0455 (mean),  2 (max)   \\ \hline
			\multicolumn{1}{|c|}{$k=2$} & \multicolumn{1}{c|}{1}      &  2.3261 (mean),  5 (max)      & \multicolumn{1}{c|}{1}      &  1.0417 (mean),  2 (max)   \\ \hline
			\multicolumn{1}{|c|}{$k=3$} & \multicolumn{1}{c|}{1}      &   1.1818 (mean),  3 (max)     & \multicolumn{1}{c|}{1}      &  1.0833 (mean),  2 (max)   \\ \hline
		\end{tabular}
		\caption{\label{table:iterations} Mean and maximal iteration numbers for polynomial degrees $p=1,2$ and diffusion coefficient numbers $k=1, 2, 3$.}
	\end{table}
	
	%%%%%%%%%%%%%%%%%%%%%%%%%%%%%%%%%%%%%%%%%%%%%
	%%%%%%%%%%%%%%%%%%%%%%%%%%%%%%%%%%%%%%%%%%%%%
	\section{Proofs}\label{section:proofs}
	%%%%%%%%%%%%%%%%%%%%%%%%%%%%%%%%%%%%%%%%%%%%%
	%%%%%%%%%%%%%%%%%%%%%%%%%%%%%%%%%%%%%%%%%%%%%

		Below we present proofs of intermediate results leading to our main Theorem~\refeq{theorem:solver} of $L$- and $p$-robust contraction of the multigrid solver and the $L$- and $p$-robust two-sided bound of the algebraic error by the built-in \textsl{a~posteriori} estimator.
		We emphasize that this result improves the recent work~\cite{Mir_Pap_Voh_lam_21} by removing the $L$-dependence. 
		From an \emph{algorithmic} point of view, this is done by applying \emph{local} smoothing only on patches which change in the refinement step on lowest-order levels instead of on \emph{every} patch as was the case in~\cite{Mir_Pap_Voh_lam_21}. From an \emph{analysis} point of view, $L$-robustness is achieved thanks to the strengthened Cauchy--Schwarz inequality on bisection-generated meshes (Proposition~\ref{lem:str_CS_adapt}) building on the property that the levelwise overlap of the smoothed patches stays uniformly bounded.
		The next essential ingredient to prove the main result is an $hp$-stable decomposition on bisection generated meshes (Proposition~\ref{lemma:multilevel_decomposition}), then one combines the results carefully together with the simple but crucial observation of uniform boundedness in the number of overlapping patches for a fixed level (Lemma~\ref{lemma:norm_estimates}) and bounds on the step-sizes and the levelwise solver update (Lemma~\ref{lemma:loc_global_stepsizes}).

	%Before we present the proofs of the main results, we consider some preparatory steps.
	\subsection{Auxiliary results}\label{subsection:auxiliary_result}
	We start with the simple observation that the number of overlapping patches is uniformly bounded.

	\begin{lemma}[Finite patch overlap]\label{lemma:norm_estimates}
		
		For all $T \in \TT_{\ell}$, there holds
		\begin{align} \label{equation:patch_overlap}
			\# (\VV_\ell \cap T) = d+1.
		\end{align}
		Therefore, for all $q \in \mathbb{N}$, it holds that
		\begin{align} \label{equation:norm_estimate_general}
			\enorm[\Big]{\sum \limits_{z \in \VV_\ell} v_{\ell, z}}^2 \le (d+1) \sum \limits_{z \in \VV_\ell} \enorm{v_{\ell, z}}^2 \quad \text{for all } v_{\ell, z} \in \V_{\ell,z}^q.
		\end{align}
		Similar arguments show that
		\begin{align}\label{equation:norm_estimate_special}
			\Big \Vert \nabla \sum \limits_{z \in \VV_\ell} v_{\ell, z} \Big \Vert^2 \le (d+1) \sum \limits_{z \in \VV_\ell} \Vert \nabla v_{\ell, z} \Vert^2 \quad \text{for all } v_{\ell, z} \in \V_{\ell,z}^q.
		\end{align}
	\end{lemma}
	
	\begin{proof}
		The overlap~\eqref{equation:patch_overlap} is clear from the geometry of the elements in the mesh.
		For all $\ell = 0, \dots, L$, the discrete Cauchy--Schwarz inequality and~\eqref{equation:patch_overlap} lead to
		\begin{align*}
			\enorm[\Big]{\sum \limits_{z \in \VV_\ell} v_{\ell, z}}^2 = \sum \limits_{T  \in \TT_\ell} \enorm[\Big]{\sum \limits_{z \in \VV_\ell \cap T} v_{\ell, z}}_T^2 \le (d+1) \sum \limits_{z \in \VV_\ell}  \enorm{v_{\ell, z}}^2.
		\end{align*}
		This concludes the proof.
	\end{proof}

	Next, we present bounds on the step-size and the levelwise solver update.
	\begin{lemma}\label{lemma:loc_global_stepsizes}
		For all $\ell \in \{1, \dots, L\}$, we have
		\begin{align}
			\enorm{\lambda_\ell \rho_\ell}^2
			\le \lambda_\ell \sum\limits_{z \in \VV_\ell^+} \enorm{\rho_{\ell,z}}^2.
			\label{equation:loc_global_stepsizes}
		\end{align}
		Moreover, we have upper and lower bounds for the step-sizes,
		\begin{align}
			\frac{1}{d+1}\le \lambda_\ell \le {d+1} \quad \text{for all } \ell = 1, \dots, L-1
			\quad \text{and} \quad
			\frac{1}{d+1}\le \lambda_L.
			\label{equation:bounds_stepsizes}
		\end{align}
	\end{lemma}
	
	\begin{proof}
		\textbf{Step 1:} Proof of \eqref{equation:loc_global_stepsizes} if $\ell = L$ or $(R_L(\rho_\ell) - \edual{\sigma_{\ell-1}}{\rho_\ell})/\enorm{\rho_\ell}^2 \le d+1$ for $\ell \in \{1, \dots, L-1\}$. From Step (ii) of Algorithm \ref{algorithm:solver}, we have that $\lambda_\ell = (R_L(\rho_\ell) - \edual{\sigma_{\ell-1}}{\rho_\ell})/\enorm{\rho_\ell}^2$ and thus
		\begin{align*}
			\enorm{\lambda_\ell \rho_\ell}^2  \!
			= \lambda_\ell \, \cfrac{R_L(\rho_\ell) - \edual{\sigma_{\ell-1}}{\rho_\ell}}{\enorm{\rho_\ell}^2} \, \enorm{\rho_\ell}^2
			&= \lambda_\ell \sum\limits_{z \in \VV_\ell^+} \left(R_L(\rho_{\ell,z}) - \edual{\sigma_{\ell-1}}{\rho_{\ell,z}} \right) \\
			&\eqreff*{equation:alg_step2}=~\lambda_\ell\! \! \sum\limits_{z \in \VV_\ell^+} \! \enorm{\rho_{\ell,z}}^2.
		\end{align*}
		\textbf{Step 2:} Proof of \eqref{equation:loc_global_stepsizes} in the remaining cases. We use the finite overlap of the patches in Lemma~\ref{lemma:norm_estimates} to obtain
		\begin{align*}
			\enorm{\lambda_\ell \rho_\ell}^2
			= \frac{\lambda_\ell}{d+1}   \enorm{\rho_\ell}^2
			\stackrel{\eqref{equation:norm_estimate_general}}\le \frac{\lambda_\ell}{d+1} (d+1)  \sum\limits_{z \in \VV_\ell^+} \enorm{\rho_{\ell,z}}^2  =\lambda_\ell \sum\limits_{z \in \VV_\ell^+} \enorm{\rho_{\ell,z}}^2.
		\end{align*}
		\textbf{Step 3:} Proof of \eqref{equation:bounds_stepsizes}. For $\ell \in \{1, \dots, L-1\}$, the upper bound is guaranteed by definition of $\lambda_\ell$. The lower bound for $\ell \in \{1, \dots, L\}$ is trivial if $\lambda_\ell = 1/(d+1)$. Otherwise, it is a consequence of the finite patch overlap:
		\begin{align*}
			\lambda_\ell
			& = \frac{R_L(\rho_\ell) -
				\edual{\sigma_{\ell-1}}{\rho_\ell}}{\enorm{\rho_\ell}^2}
			\stackrel{\eqref{equation:alg_step2}}=
			\frac{ \sum_{z \in \VV_\ell^+} \enorm{\rho_{\ell,z}}^2}{\enorm{\rho_\ell}^2}
			\stackrel{\eqref{equation:norm_estimate_general}}\ge \frac{1}{d+1}.
		\end{align*}
		This concludes the proof.
	\end{proof}
	%%%%%%%%%%
	
	In the next two subsections, we combine existing results from the literature to obtain a multilevel $hp$-robust stable decomposition and a strengthened Cauchy--Schwarz inequality for our setting of bisection-generated meshes. These will be crucial for the proofs of Theorem~\ref{theorem:solver} and Corollary~\ref{corollary:solver} in Subsection~\ref{subsection:proof_main_result} below.
	
	%%%%%%%%%%%%%%%%%%%%%%%%%%%%%%%%%%%%%%%%%%%%%
	\subsection{Multilevel \textit{hp}-robust stable decomposition on NVB-generated meshes} \label{subsection:SD}
	%%%%%%%%%%%%%%%%%%%%%%%%%%%%%%%%%%%%%%%%%%%%%
	
	We start by recalling the one-level $p$-robust stable decomposition from Section~3.4 and Section~4.3 in \cite{Sch_Mel_Pec_Zag_08} for \(d=2\) and \(d=3\), respectively.
	\begin{lemma}[$p$-robust one level decomposition]
		\label{lemma:one_level_p}
		Let $v_L \in \V_L^p$. Then, there exists a decomposition
		\begin{align} \label{equation:one_level_p_decomposition}
			&v_L =  v_L^1 + \sum_{z \in \VV_{L}} v^p_{L,z} \quad \text{with } v_L^1 \in \V_L^1 \text{ and } v^p_{L,z} \in \V_{L,z}^p,
		\end{align}
		which is stable in the sense of
		\begin{align} \label{equation:one_level_p_stability}
			& \| \nabla v_L^1\|^2 + \sum_{z \in \VV_{L}} \| \nabla v^p_{L,z} \|^2  \le  \cps^2 \| \nabla v_L \|^2.
		\end{align}
		The constant $\cps$ depends only on the space dimension~$d$, the $\gamma$-shape regularity \eqref{equation:shape_regularity}, and the quasi-uniformity constant~$\cqu$ from \eqref{equation:coarse_quasiuniformity}.
	\end{lemma}
	
	Similarly, we recall the \emph{local} multilevel decomposition for piecewise affine functions proven in~\cite[Lemma~3.1]{Wu_Zheng_Multigrid}. In order to present this stable decomposition in a form that is more suitable for our forthcoming analysis, we add a short proof for completeness.
	
	\begin{lemma}[$h$-robust local multilevel decomposition for lowest-order functions]\label{lemma:local_multilevel_decomposition}
		Let $v_L^1 \in \V_L^1$. Then, there exists a decomposition
		\begin{align}\label{equation:multilevel_local_lowest_order_decomposition}
			&v_L^1 = \sum_{\ell=0}^L \sum_{z \in \VV_{\ell}^+} v^1_{\ell,z} \quad \text{with } v^1_{\ell,z} \in \V^1_{\ell,z},
		\end{align}
		which is stable in the sense of
		\begin{align} \label{equation:multilevel_local_lowest_order_stability}
			&\sum_{\ell=0}^{L} \sum_{z \in \VV_{\ell}^+} \| \nabla v^1_{\ell,z} \|^2  \le  \cs^2 \| \nabla v_L^1 \|^2.
		\end{align}
		The constant $\cs$ depends only on the space dimension~$d$, the $\gamma$-shape regularity \eqref{equation:shape_regularity}, and the quasi-uniformity constant~$\cqu$ from \eqref{equation:coarse_quasiuniformity}.
	\end{lemma}
	
	\begin{proof}
		Let $v_L^1 \in \V_L^1$. Define $w_\ell^1 \coloneqq (\Pi_\ell - \Pi_{\ell-1}) v_L^1$ for $ \ell \in \{0, \dots, L \}$, where $\Pi_{-1} \coloneqq 0$ and $\Pi_\ell$ is the projection to $\V_{\ell}^1$ from~\cite[Section~3]{Wu_Zheng_Multigrid}.
		From~\cite[Lemma~3.1]{Wu_Zheng_Multigrid}, it holds that $w_\ell^1 \in {\rm span} \set{\varphi_{\ell,z}}{z \in \VV_\ell^+ }$ with $\varphi_{\ell, z}$ being the $\S^1(\TT_\ell)$ hat-function at vertex $z\in \VV_{\ell}$. We decompose $w_\ell^1 = \sum_{z \in \VV_{\ell}^+} v_{\ell, z}^1$ with $v_{\ell,z}^1 := w_\ell^1 (z) \varphi_{\ell,z} \in \V^1_{\ell, z}$ and thus obtain
		\begin{align}\label{equation:one_level_p1_decomposition}
			v_L^1 =  \sum_{\ell=0}^{L} (\Pi_\ell - \Pi_{\ell-1}) v_L^1 =  \sum_{\ell=0}^{L} w_\ell^1 =  \sum_{\ell=0}^{L} \sum_{z \in \VV_{\ell}^+}  v^1_{\ell,z}.
		\end{align}
		For fixed $\ell$ and $z \in \VV_{\ell}^+$, the equivalence of norms on finite-dimensional spaces proves
		\begin{align}\label{equation:loc_nd_L2}
			\begin{split}
				\| v^1_{\ell,z}\|_{\omega_{\ell,z}}
				&\le \sum_{T \in {\TT_{\ell,z}}}  \| w^1_{\ell} (z) \varphi_{\ell,z} \|_T \\
				&\le \! \! \sum_{T \in {\TT_{\ell,z}}} \|w_{\ell}^1 \|_{L^{\infty}(T)} |T|^{1/2}
				\lesssim \sum_{T \in {\TT_{\ell,z}}} \|w_{\ell}^1 \|_T \simeq \|w_{\ell}^1 \|_{\omega_{\ell,z}},
			\end{split}
		\end{align}
		where the hidden constants depend only on $\gamma$-shape regularity \eqref{equation:shape_regularity}.
		To obtain stability of the decomposition~\eqref{equation:one_level_p1_decomposition}, we use an inverse inequality on the patches
		and the stability proved in~\cite[Lemma~3.7]{Wu_Zheng_Multigrid}:
		\begin{align*}
			\sum_{\ell=0}^{L} \sum_{z \in \VV_{\ell}^+} \| \nabla v^1_{\ell,z} \|^2
			&\lesssim  \sum_{\ell=0}^{L} \sum_{z \in \VV_{\ell}^+} h_{\ell,z}^{-2} \| v^1_{\ell,z} \|^2_{\omega_{\ell,z}}
			\stackrel{\eqref{equation:loc_nd_L2}}\lesssim  \sum_{\ell=0}^{L} \sum_{z \in \VV_{\ell}^+} h_{\ell,z}^{-2} \| w^1_\ell \|^2_{\omega_{\ell,z}}  \\
			&= \sum_{\ell=0}^{L} \sum_{z \in \VV_{\ell}^+} h_{\ell,z}^{-2} \|  (\Pi_\ell - \Pi_{\ell-1}) v_L^1 \|^2_{\omega_{\ell,z}}
			\stackrel{\text{\cite{Wu_Zheng_Multigrid}}}\lesssim  \| \nabla v_L^1 \|^2.
		\end{align*}
		This concludes the proof.
	\end{proof}

	The combination of the two previous lemmas, done similarly in~\cite[Proposition 7.6]{Mir_Pap_Voh_19} for a \emph{non-local} and hence not $h$-robust solver, leads to the following $hp$-robust decomposition.
	\begin{proposition}[$hp$-robust local multilevel decomposition]\label{lemma:multilevel_decomposition}
		Let $v_L \in \V_L^p$. Then, there exist $v_0 \in \V_0^1, \ v_{\ell,z} \in \V_{\ell,z}^1$, and $v_{L,z} \in \V_{L,z}^p$ such that
		\begin{align} \label{equation:multilevel_local_decomposition}
			&v_L = v_0 + \sum_{\ell=1}^{L-1} \sum_{z \in \VV_{\ell}^+}  v_{\ell,z} + \sum_{z \in \VV_{L}}  v_{L,z}.
		\end{align}
		and this decomposition is stable in the sense of
		\begin{align} \label{equation:multilevel_local_stability}
			&\enorm{v_0}^2 +  \sum_{\ell=1}^{L-1} \sum_{z \in \VV_{\ell}^+} \enorm{v_{\ell,z}}^2
			+ \sum_{z \in \VV_L} \enorm{v_{L,z}}^2 \le  \csdsqu \enorm{v_L}^2.
		\end{align}
		The constant $ \csd \ge 1 $ depends only on the space dimension~$d$, $\gamma$-shape regularity \eqref{equation:shape_regularity}, the quasi-uniformity constant~$\cqu$ from \eqref{equation:coarse_quasiuniformity}, and the ratio of $\Lambda_{\max}$ and $\Lambda_{\min}$.
	\end{proposition}
	\begin{proof}
		Let $v_L \in \Vp_L$. We begin with the decomposition of $v_L$ by \eqref{equation:one_level_p_decomposition}, then continue with the further decomposition of the lowest-order contribution $v_L^1$ in a multilevel way \eqref{equation:multilevel_local_lowest_order_decomposition}:
		\begin{align*}
			v_L
			\stackrel{\eqref{equation:one_level_p_decomposition}}=  v_L^1 + \sum_{z \in \VV_{L}} v^p_{L,z}
			&\stackrel{\eqref{equation:multilevel_local_lowest_order_decomposition}}=
			\sum_{\ell=0}^L \sum_{z \in \VV_{\ell}^+} v^1_{\ell,z} + \sum_{z \in \VV_{L}} v^p_{L,z}\\
			&= \sum_{z \in \VV_0} v^1_{0,z} + \sum_{\ell=1}^{L-1} \sum_{z \in \VV_{\ell}^{+}} v^1_{\ell,z}
			+ \sum_{z \in \VV_L^+} v^1_{L,z}
			+ \sum_{z \in \VV_{L}} v_{L,z}^p.
		\end{align*}
		By defining $v_0 := \sum_{z \in \VV_0} v^1_{0,z} \in \V_0^1$, $v_{\ell,z} := v^1_{\ell,z} \in \V_{\ell,z}^1$ for $z \in \VV_{\ell}^+$ and $  1 \le \ell \le L-1$, and $ v_{L,z} := v^1_{L,z}+ v^p_{L,z} \in \V_{L,z}^p$ for $z \in \VV_{L}^+$ and $ v_{L,z} :=  v^p_{L,z} \in \V_{L,z}^p$ for $z \in \VV_{L} \setminus \VV_{L}^+$, we obtain the decomposition~\eqref{equation:multilevel_local_decomposition}. It remains to show that this decomposition is stable~\eqref{equation:multilevel_local_stability}. First, we have for the coarsest level that
		\begin{align*}
			\Vert \nabla v_0 \Vert^{2} \eqreff{equation:norm_estimate_special}\le (d+1) \sum_{z \in \VV_0} \Vert \nabla  v^1_{0,z}  \Vert^{2}.
		\end{align*}
		For the finest level, it holds that
		\begin{align*}
			\sum_{z \in \VV_L} \| \nabla v_{L,z} \|^2 &\le \sum_{z \in \VV_{L} \backslash \VV_{L}^+} \| \nabla v^p_{L,z}  \|^2 + 2 \, \sum_{z \in \VV_L} \bigl(\| \nabla  v^1_{L,z}  \|^2 +  \| \nabla v^p_{L,z}  \|^2 \bigr) 
			\\
			&\le (d+1) \sum_{z \in \VV_L^+} \| \nabla  v^1_{L,z}  \|^2 + (d+1)  \sum_{z \in \VV_L} \| \nabla v^p_{L,z}  \|^2.
		\end{align*}
		A combination of the two estimates shows that
		\begin{align*}
			\| \nabla &v_0 \|^2 +  \sum_{\ell=1}^{L-1} \sum_{z \in \VV_{\ell}^{+}} \| \nabla v_{\ell,z} \|^2
			+ \sum_{z \in \VV_L} \| \nabla v_{L,z}\|^2 \\
			&\le (d+1) \Big(\sum_{z \in \VV_0} \Vert \nabla  v^1_{0,z}  \Vert^{2} + \sum_{\ell=1}^{L-1} \! \sum_{z \in \VV_{\ell}^{+}} \! \| \nabla v^1_{\ell,z} \|^2 + \sum_{z \in \VV_L^+} \| \nabla  v^1_{L,z}  \|^2 + \sum_{z \in \VV_L} \| \nabla v^p_{L,z}  \|^2\Big) \\
			&\le (d+1) \sum_{\ell=0}^{L} \sum_{z \in \VV_{\ell}^{+}} \| \nabla v^1_{\ell,z} \|^2
			+ (d+1)\sum_{z \in \VV_{L}} \| \nabla v^p_{L,z} \|^2\\
			&\eqreff*{equation:multilevel_local_lowest_order_stability}\le  \cs^{2} (d+1) \| \nabla v^1_L\|^2
			+ (d+1)\sum_{z \in \VV_{L}} \| \nabla v^p_{L,z} \|^2 \\
			&\eqreff*{equation:one_level_p_stability}\le \max\{1, \cs^{2}\} \cps^{2} (d+1)
			\| \nabla v_{L} \|^2.
		\end{align*}
		Hence, the decomposition~\eqref{equation:multilevel_local_decomposition} is stable with $(\csd')^2 \coloneqq \max\{1, \cs^{2}\} \cps^{2} (d+1)$ with respect to the $H^1(\Omega)$-seminorm.
		Taking into account the variations of the diffusion coefficient $\K$, we obtain~\eqref{equation:multilevel_local_stability}
		with the stability constant $\csd \coloneqq \csd' \Lambda_{\max} / \Lambda_{\min}$.
	\end{proof}

	%%%%%%%%%%%%%%%%%%%%%%%%%%%%%%%%%%%%%%%%%%%%%
	\subsection{Strengthened Cauchy--Schwarz inequality on NVB-generated meshes} \label{subsection:scs}
	%%%%%%%%%%%%%%%%%%%%%%%%%%%%%%%%%%%%%%%%%%%%%
	The following results are proved in the spirit of \cite{Hipt_Wu_Zheng_cvg_adpt_MG_12, CNX_12}. Note that the setting of this work is similar to \cite{Hipt_Wu_Zheng_cvg_adpt_MG_12}, and unlike \cite{CNX_12}, the underlying adaptive meshes of the space hierarchy are not restricted to one bisection per level.
	
	For analysis purposes, we introduce a sequence of uniformly refined triangulations indicated by $\{\widehat{\TT}_j\}_{j=0}^M$ such that 
		\(\widehat{\TT}_{j+1} \coloneqq \texttt{refine}(\widehat{\TT}_j, \widehat{\TT}_{j})\) and \(\widehat{\TT}_0 = \TT_{0}\), where \(\mathtt{refine}\) enforces one bisection per element. According to \cite{stevenson2008}, admissibility of \(\TT_{0}\) ensures that indeed each element \(T \in \widehat{\TT}_j\) is bisected only once into two children \(T^{\prime}, T^{\prime \prime} \in \widehat{\TT}_{j+1}\).
	%for all \(\ell \in \{1, \ldots, L\}\) there exists \( j \in \{1, \ldots, M\}\) 
	%such that \(\TT_{\ell} \subseteq \widehat{\TT}_j\), \(\widehat{\TT}_0 = \TT_{0}\), and the uniform hierarchy of meshes is nested. This uniform triangulation can be obtained similarly to \cite{CNX_12} via successive mesh coarsening.}
	% Starting with $\widehat{\TT}_0 \coloneqq \TT_0$, the mesh $\widehat{\TT}_\ell$ of uniform mesh size $\widehat{h}_\ell \coloneqq \max_{T  \in \widehat{\TT}_{\ell}} h_T$ is obtained by uniformly refining $\widehat{\TT}_{\ell - 1}$, i.e., every element $T \in \widehat{\TT}_{\ell-1}$ is successively bisected into $2^d$ child elements
	%$T^\prime \in \widehat{\TT}_\ell$ with measure $\vert T^\prime \vert = 2^{-d} \vert T \vert$; cf. \cite[Theorem~2.1]{stevenson2008}. 
	In the following, we will indicate the equivalent notation to Section~\ref{section:Setting} on uniform triangulations
	$\widehat{\TT}_{j}$ with a hat, e.g., $\widehat{\V}_{j}^1$ is the equivalent of $\V_\ell^1$ on the uniformly refined mesh $\widehat{\TT}_{j}$.
	The connection of the uniformly refined meshes and their adaptively generated counterpart requires further notation. For a given level $0 \le \ell \le L$ and a given node $z\in \VV_\ell$, we define the generation  $g_{\ell, z}$ of the patch by the maximum number of times an element of the patch has been bisected
	\begin{align} \label{equation:generation}
	g_{\ell,z} := \max_{T \in \TT_{\ell, z}}  \log_2(|T_0|/|T|) \in \N_0,
	\end{align}
	where $T_0 \in \TT_0$ denotes the unique ancestor element of $T \in \TT_\ell$. Define the maximal generation $M = \max_{z\in \VV_L} g_{L,z}$.
	
	First, we present the following result for uniformly refined meshes and then exploit this for our setting of adaptively refined meshes.

	\begin{lemma}[Strengthened Cauchy--Schwarz on nested uniform meshes]\label{lem:str_CS_unif}
	Let $0\le i \le j \le M$, and $\widehat u_i \in \widehat \V_i^1$ as well as $\widehat v_j \in \widehat \V_j^1$.
	Then, it holds that
	\begin{align}\label{equation:str_CS_unif}
		\edual{\widehat u_i}{\widehat v_j} \le \cscsunif\, \delta^{j-i} \widehat h_j^{-1}
		\big\|  \nabla \widehat u_i  \big\| \big\| \widehat v_j  \big\|,
	\end{align}
	where $\delta=2^{-1/2}$ and $\cscsunif>0$ depends only on the domain $\Omega$, the initial triangulation $\TT_0$, $\Lambda_{\max}$, \linebreak[4] \(\max_{T  \in \widehat{\TT}_{M}}\Vert \div(\K) \Vert_{L^{\infty}(T)}\), and $\gamma$-shape regularity from \eqref{equation:shape_regularity}.
	\end{lemma}
	
	\begin{proof}
	We begin by splitting the domain $\Omega$ into elementwise components, applying integration by parts, and using the Cauchy--Schwarz inequality. Note that the restriction of $\widehat u_i$ to any element $T \in \widehat \TT_i$ is an affine function, and hence the second derivatives vanish. Thus, it holds with the outer normal $\boldsymbol{n}$  to $\partial T$ that
	\begin{align*}
		\edual{\widehat u_i}{\widehat v_j}
		&= \sum_{T \in \widehat \TT_i} \int_T  \K \nabla \widehat u_i \cdot \nabla \widehat v_j \, \dx \\
		&=  \sum_{T \in \widehat \TT_i} \Big( - \int_T {\rm div} (\K \nabla \widehat u_i) \widehat v_j \, \dx +  \int_{\partial T}  \K \nabla \widehat u_i \cdot \boldsymbol{n} \, \widehat v_j \, \dx \Big)\\
		&\le  \sum_{T \in \widehat \TT_i} \Big(  \big\| ({\rm div} \K)  \cdot \nabla \widehat u_i  \big\|_{L^2(T)}  \big\| \widehat v_j  \big\|_{L^2(T)}  +  \big\|  \K \nabla \widehat u_i  \big\|_{L^2(\partial T)}   \big\| \widehat v_j  \big\|_{L^2(\partial T)}  \Big).
	\end{align*}
	Due to $\K \in W^{1, \infty}(T)$, the fact that $\widehat{u}_i, \widehat{v}_j$ are piecewise affine, a discrete trace inequality, and $\widehat h_i^{-1}  \gtrsim 1$, we get
	\begin{align*}
		\edual{\widehat u_i}{\widehat v_j} &\lesssim \sum_{T \in \widehat \TT_i} \Big(  \big\|\nabla \widehat u_i  \big\|_{L^2(T)}  \big\| \widehat v_j  \big\|_{L^2(T)}  +  \big\|  \nabla \widehat u_i  \big\|_{L^2(\partial T)}   \big\| \widehat v_j  \big\|_{L^2(\partial T)}  \Big)\\
		&\lesssim  \sum_{T \in \widehat \TT_i} \Big(  \big\|\nabla \widehat u_i  \big\|_{L^2(T)}  \big\| \widehat v_j  \big\|_{L^2(T)}  + \bigl(\widehat  h_i^{-1/2} \big\|  \nabla \widehat u_i  \big\|_{L^2(T)}\bigr) \, \bigl(\widehat h_i^{-1/2} \big\| \widehat v_j  \big\|_{L^2(T)}\bigr)  \Big)\\
		&=  \sum_{T \in \widehat \TT_i} \bigl(1+h_i^{-1}\bigr) \, \big\|\nabla \widehat u_i  \big\|_{L^2(T)}  \big\| \widehat v_j  \big\|_{L^2(T)}
		\\
		&\lesssim  \sum_{T \in \widehat \TT_i} \widehat h_i^{-1} \big\|  \nabla \widehat u_i  \big\|_{L^2(T)} \big\| \widehat v_j  \big\|_{L^2(T)}.
	\end{align*}
	Moreover, note that due to uniform refinement, we have equivalence $\delta^{j-i} = (2^{-1/2})^{j-i} \simeq \big(\widehat h_j/\widehat h_i \big)^{1/2}$ and $\widehat{h}_j \le \widehat{h}_i$. Using the last equation multiplied by $1 = \widehat{h}_j^{1/2} \, \widehat{h}_j^{-1/2}$, we derive that
	\begin{align*}
		\edual{\widehat u_i}{\widehat v_j}
		&\lesssim \sum_{T \in \widehat \TT_i} \Big(\cfrac{\widehat h_j}{\widehat h_i}\Big)^{1/2} \, \widehat{h}_i^{-1/2} \, \widehat{h}_j^{-1/2} \big\|  \nabla \widehat u_i  \big\|_{L^2(T)} \big\| \widehat v_j  \big\|_{L^2(T)} \\
		&\lesssim \sum_{T \in \widehat \TT_i} \delta^{j-i} \, \widehat h_j^{-1} \, \big\|  \nabla \widehat u_i  \big\|_{L^2(T)} \big\| \widehat v_j  \big\|_{L^2(T)} \le  \widehat h_j^{-1} \delta^{j-i} \big\|  \nabla \widehat u_i  \big\|_{L^2(\Omega)} \big\| \widehat v_j  \big\|_{L^2(\Omega)}.
	\end{align*}
	This concludes the proof.
	\end{proof}

	The last result enables us to tackle the setting of adaptively-refined meshes.

	\begin{proposition}[Strengthened Cauchy--Schwarz on nested adaptive meshes]\label{lem:str_CS_adapt}
	Consider levelwise functions $ v_\ell = \sum_{z \in \VV_\ell^+} v_{\ell, z}^1  \in \V_{\ell}^1$ with $v_{\ell, z}^1 \in \V_{\ell, z}^1$ for all $1 \le \ell \le L-1$.
	Then, it holds that
	\begin{align}\label{equation:strengthened_CS_adapt}
		\sum_{\ell=1}^{L-1} \sum_{k=1}^{\ell-1}	\edual{v_k}{ v_\ell} \le \cscs
		\Big(	\sum_{k=1}^{L-2} \sum_{w \in \VV_k^+}	\enorm{ v_{k,w}^1 }^2 \Big)^{1/2}  \Big( \sum_{\ell=1}^{L-1}  \sum_{z \in \VV_\ell^+} \enorm{  v_{\ell,z}^1 }^2 \Big)^{1/2},
	\end{align}
	where $\cscs>0$ depends only on $\Omega$, the initial triangulation $\TT_0$, $\Lambda_{\max} / \Lambda_{\min}$, \linebreak[4] \(\max_{T  \in \TT_{L}}\Vert \div(\K) \Vert_{L^{\infty}(T)} / \Lambda_{\min}\), and $\gamma$-shape regularity \eqref{equation:shape_regularity}.	\end{proposition}

	\begin{proof}
	Let $M \in \mathbb{N}$. The proof consists of five steps.
	
	\medskip
	\noindent{\bf Step~1.}\quad
	First note that, for any $0 < \delta < 1$ and $x_i, y_i > 0$ with $0\le i \le M$, there holds
	\begin{align}\label{geo_series}
		\sum_{i=0}^M \sum_{j=i}^M \delta^{j-i} x_i y_j \le \frac{1}{1- \delta} \Big( \sum_{i=0}^M x_i^2 \Big)^{1/2} \Big( \sum_{j=0}^M y_j^2 \Big)^{1/2}.
	\end{align}
	To see this, we change the summation order accordingly and use the Cauchy--Schwarz inequality to obtain
	\begin{align*}
		\sum_{i=0}^M \sum_{j=i}^M \, &\delta^{j-i} x_i y_j
		= \! \sum_{i=0}^M \sum_{m=0}^{M-i}  \delta^{m} x_i y_{m+i}
		= \sum \limits_{m=0}^M \sum \limits_{i=0}^{M-m} \delta^m x_i y_{m+i} \\
		&\le \sum_{m=0}^{M} \delta^{m}\, \Big[ \big( \sum_{i=0}^{M-m} x_i^2 \big)^{1/2} 
		\big( \sum_{i=0}^{M-m} y_{m+i}^2 \big)^{1/2} \Big]
		\le  \Big(\sum_{m=0}^{M} \delta^{m} \Big) \Big( \sum_{i=0}^{M} x_i^2 \Big)^{1/2} \Big(\sum_{j=0}^{M} y_{j}^2  \Big)^{1/2}.
	\end{align*}
	The geometric series then proves the claim~\eqref{geo_series}.

	\medskip
	\noindent{\bf Step~2.}\quad
	Let $z \in \VV_L$ and $0 \le j \le M$ and recall the patch generation $g_{\ell,z} $ from~\eqref{equation:generation}. We introduce the set
	\begin{align}\label{overlap_gens}
		\mathscr{L}_{\underline\ell,\overline\ell}(z, j) := \{ \ell \in \{\underline\ell, \ldots, \overline\ell \}: z \in \VV_\ell^+ \text{ and } g_{\ell,z} = j \} \quad \text{with } 0 \le \underline{\ell} \le \overline{\ell} \le L.
	\end{align}
	This set allows to track how large the levelwise overlap of patches with the same generation is.
	Crucially, the cardinality of these sets is uniformly bounded by
	\begin{align}\label{bound_overlap_gens}
		\max_{\substack{z \in \VV_L \\ 0 \le j \le M}}\#(	\mathscr{L}_{0,L}(z, j)) \le C_{\mathrm{lev}} < \infty;
	\end{align}
	see, e.g., \cite[Lemma~3.1]{Wu_Chen_cvg_adpt_MG_06} in the two-dimensional setting with arguments that transfer to three dimensions. The constant $C_{\mathrm{lev}}$ solely depends on $\gamma$-shape regularity~\eqref{equation:shape_regularity}.
	
	\medskip
	\noindent{\bf Step~3.}\quad
	We introduce a way to reorder the patch contributions by generations~\eqref{equation:generation}. Note that, for any $0\le j \le M$, $1 \le \ell \le L-1$, and $z \in \VV_\ell^+ $ such that $g_{\ell,z} = j$, the patch contribution $ v_{\ell,z}^1 \in  \V_{\ell, z}^1$ also belongs to
	$\widehat{\V}_j^1$. Once the generation constraint is introduced, one can shift the perspective from summing over ``adaptive'' levels and associated vertices to summing over ``uniform'' vertices and \emph{only} the (finitely many, cf.~\eqref{bound_overlap_gens}) levels where each vertex satisfies the generation constraint, i.e., for $0\le \underline\ell \le \overline\ell \le L $ and $0 \le j \le M$, the two following sets coincide
	
	\begin{align}
		\begin{split}
			\set{(\ell, z) \in \N_0 \times \VV_L}{\ell \in &\{\underline\ell, \ldots, \overline\ell\},  z \in \VV_\ell^+ \text{ with } g_{\ell,z} = j}
			\\
			&= \set{(\ell, z)\in \N_0 \times \VV_L}{z \in \widehat{\VV}_j, \ell \in  \! {\mathscr{L}}_{\underline\ell, \overline\ell}(z, j)}.
		\end{split}
		\label{unif_adapt}
	\end{align}

	\medskip
	\noindent{\bf Step~4.} According to $\gamma$-shape regularity \eqref{equation:shape_regularity}, all elements in the patch have comparable size depending on \(C_{\mathrm{qu}}\) from \eqref{equation:coarse_quasiuniformity}. If $g_{\ell, z} = j$, (at least) one element $T^\star \in \TT_{\ell, z}$ satisfies $T^\star \in \widehat{\TT}_j$ and it follows that $ \widehat h_j \simeq \vert T^\star \vert^{1/d} \simeq \vert \omega_{\ell, z} \vert^{1/d} \simeq h_{\ell, z}.$
	In particular, there exists $C_{\rm eq}>0$ such that
	\begin{align}
		\widehat h_j^{-1} \le C_{\rm eq} h_{\ell, z}^{-1}.
		\label{equation:loc_size_equiv}
	\end{align}

	\medskip
	\noindent{\bf Step~5.}\quad
	We proceed to prove the main estimate~\eqref{equation:strengthened_CS_adapt}. The central feature of the following approach is to introduce \emph{additional} sums over the generations with generation constraints, i.e., there holds for every admissible $\ell, k$, that
	\begin{align*}
		\edual{v_k}{ v_\ell}
		&=\sum_{z \in \VV_\ell^+}  \sum_{w \in \VV_k^+} \edual{v_{k,w}^1}{v_{\ell,z}^1}
		%\\
		%&
		= \sum_{j=0}^{M} \sum_{i=0}^{M}  \sum_{\substack{z \in \VV_\ell^+ \\ g_{\ell,z} = j}} \sum_{\substack{w \in \VV_k^+ \\ g_{k,w} = i}} \edual{v_{k,w}^1}{v_{\ell,z}^1}
		\\
		&=  \sum_{j=0}^{M} \sum_{i=0}^{j}   \sum_{\substack{z \in \VV_\ell^+ \\ g_{\ell,z} = j}} \sum_{\substack{w \in \VV_k^+ \\ g_{k,w} = i}} \edual{v_{k,w}^1}{v_{\ell,z}^1} + \sum_{j=0}^{M}  \sum_{i=j+1}^{M}  \sum_{\substack{z \in \VV_\ell^+ \\ g_{\ell,z} = j}} \sum_{\substack{w \in \VV_k^+ \\ g_{k,w} = i}} \edual{v_{k,w}^1}{v_{\ell,z}^1}.
	\end{align*}
	We abbreviate the terms as $S_1(\ell, k)$ and $S_2(\ell, k)$, respectively. A change of the summation of order $i$ and $j$ yields for $S_1(\ell, k)$ that
	\begin{align*}
		S_1(\ell, k) = \sum_{i=0}^{M} \sum_{j=i}^{M}  \sum_{\substack{z \in \VV_\ell^+ \\ g_{\ell,z} = j}} \sum_{\substack{w \in \VV_k^+ \\ g_{k,w} = i}} \edual{v_{k,w}^1}{v_{\ell,z}^1}.
	\end{align*}
	Summing $S_2(\ell, k)$ over all $\ell$ and $k$ and changing the order of summation, we obtain
	\begin{align*}
		\sum \limits_{\ell=1}^{L-1} \sum \limits_{k=1}^{\ell-1} S_2(\ell, k) = \sum_{j=0}^{M}  \sum_{i=j+1}^{M}  \sum_{k=1}^{L-2} \sum_{\ell=k+1}^{L-1}  \sum_{\substack{z \in \VV_\ell^+ \\ g_{\ell,z} = j}} \sum_{\substack{w \in \VV_k^+ \\ g_{k,w} = i}}
		\edual{v_{k,w}^1}{v_{\ell,z}^1}.
	\end{align*}
	Combining these two identities with \eqref{unif_adapt}, we see that
	\begin{align*}
		\sum \limits_{\ell=1}^{L-1} \sum \limits_{k=1}^{\ell-1} \Big(S_1(\ell, k) + S_2(\ell, k) \Big) &= \sum_{i=0}^{M} \sum_{j=i}^{M}  \sum_{\ell=1}^{L-1} \edual[\Big]{\sum_{\substack{w \in \widehat{\VV}_i}}  \sum_{k \in \mathscr{L}_{1, \ell-1}(w,i)} \!\!\! v_{k,w}^1}{\sum_{\substack{z \in \VV_\ell^+ \\ g_{\ell,z} = j}} \!\! v_{\ell,z}^1}
		\\
		&\quad + \sum_{j=0}^{M}  \sum_{i=j+1}^{M}  \sum_{k=1}^{L-2}
		\edual[\Big]{ \sum_{\substack{w \in \VV_k^+ \\ g_{k,w} = i}} \!\! v_{k,w}^1}{\sum_{\substack{z \in \widehat{\VV}_j}}  \sum_{\ell \in \mathscr{L}_{k+1,L-1}(z,j)} \!\!\! v_{\ell,z}^1}.
	\end{align*}
	We define the last two terms as $S_1$ and $S_2$, respectively. Since the second term $S_2$ is treated in the same way, we only present detailed estimations of the first term $S_1$. The strengthened Cauchy--Schwarz inequality~\eqref{equation:str_CS_unif} for functions defined on uniform meshes followed by the patch overlap~\eqref{equation:norm_estimate_general} leads to
	\begin{align*}
		S_1 \le \cscsunif
		\sum_{i=0}^{M} \sum_{j=i}^{M}  \delta^{j-i}  \sum_{\ell=1}^{L-1}
		\Big( (d+1)\sum_{w \in \widehat\VV_i}  \big\|  \! \! \! \sum_{k \in \mathscr{L}_{1, \ell-1}(w,i)} \! \! \! \nabla v_{k,w}^1 \big\|^2\Big)^{1/2}  \sum_{\substack{z \in \VV_\ell^+ \\ g_{\ell,z} = j}}   \widehat h_j^{-1} \big\| v_{\ell,z}^1 \big\|.
	\end{align*}
	The identity~\eqref{unif_adapt} and the finite levelwise overlap~\eqref{bound_overlap_gens} show
	%bound on the cardinality of $\mathscr{L}_{1, \ell-1}(w,i)$ from
	\begin{align*}
		\sum_{w \in \widehat\VV_i} \big\| \! \! \! \sum_{k \in \mathscr{L}_{1, \ell-1}(w,i)} \! \! \!  \! \! \!  \! \!  \! \! \nabla v_{k,w}^1 \big\|^2
		\le
		\sum_{k=1}^{\ell -1}  \sum_{\substack{w \in \VV_k^+ \\ g_{k,w} = i}}  \! \! \!
		\#( \mathscr{L}_{1, \ell-1}(w,i)) \big\|  \nabla  v_{k,w}^1 \big\|^2  \!
		\le  C_{\mathrm{lev}} \! \!\sum_{k=1}^{L-2}  \sum_{\substack{w \in \VV_k^+ \\ g_{k,w} = i}}  \! \! \big\|  \nabla  v_{k,w}^1 \big\|^2\! \! .
	\end{align*}
	The equivalence of mesh sizes from~\eqref{equation:loc_size_equiv} and a Poincaré-inequality prove
	\begin{align*}
		\sum_{\ell=1}^{L-1}	\sum_{\substack{z \in \VV_\ell^+ \\ g_{\ell,z} = j}}   \widehat h_j^{-1} \big\| v_{\ell,z}^1 \big\| \le C_{\rm eq} C_{\rm P} \sum_{\ell=1}^{L-1}  \!\! \sum_{\substack{z \in \VV_\ell^+ \\ g_{\ell,z} = j}}  \big\|  \nabla  v_{\ell,z}^1 \big\|.
	\end{align*}
	A combination of~\eqref{unif_adapt} with~\eqref{equation:norm_estimate_special} and~\eqref{bound_overlap_gens}, followed again by~\eqref{unif_adapt}, yields
	\begin{align*}
			\Bigl(\sum_{\ell=1}^{L-1}  \!\! \sum_{\substack{z \in \VV_\ell^+ \\ g_{\ell,z} = j}} \!\!  \big\|  \nabla  v_{\ell,z}^1 \big\|\Bigr)^2
			= \Big(\sum_{\substack{z \in \widehat\VV_j}} \! \sum_{\ell \in  \mathscr{L}_{1, L-1}(\revision{z, j}) }  \!\! \!\! \!\!  \!\!  \big\|  \nabla  v_{\ell,z}^1 \big\| \Big)^{2}
			\le (d+1) \, C_{\mathrm{lev}} \sum_{\ell=1}^{L-1}  \sum_{\substack{z \in \VV_\ell^+ \\ g_{\ell,z} = j}} \!\!    \big\|  \nabla v_{\ell,z}^1 \big\|^2.
	\end{align*}
	Thus, we obtain the bound
	\begin{equation*}
		\sum_{\ell=1}^{L-1}  \!\! \sum_{\substack{z \in \VV_\ell^+ \\ g_{\ell,z} = j}} \!\!  \big\|  \nabla  v_{\ell,z}^1 \big\| \le (d+1)^{1/2} C_{\mathrm{lev}}^{1/2} \, \Bigl(\sum_{\ell=1}^{L-1}  \sum_{\substack{z \in \VV_\ell^+ \\ g_{\ell,z} = j}} \!\!    \big\|  \nabla v_{\ell,z}^1 \big\|^2\Bigr)^{1/2}
	\end{equation*}
	Combining all estimates, together with the geometric series bound~\eqref{geo_series}, confirms
	\begin{align*}
		S_1 \le \cscsunif (d+1)
		C_{\mathrm{lev}} C_{\rm eq}  C_{\rm P} \frac{1}{1-\delta}
		\Big( \sum_{k=1}^{L-2}  \sum_{w \in \VV_k^+} \big\|  \nabla  v_{k,w}^1 \big\|^2\Big)^{1/2}
		\Big(  \sum_{\ell=1}^{L-1}   \sum_{z \in \VV_\ell^+}   \big\|  \nabla v_{\ell,z}^1 \big\|^2\Big)^{1/2}.
	\end{align*}
	Finally, the result~\eqref{equation:strengthened_CS_adapt} is obtained after summing together with the analogous estimations coming from the remaining term $S_2$ and taking into consideration the variations of the diffusion coefficient $\K$ so that the result holds for the energy norm. This concludes the proof.
	\end{proof}
	
	\subsection{Proof of the main results}\label{subsection:proof_main_result}
	For the sake of a concise presentation, we only consider the case $p> 1$. The case $p=1$ is already covered in the literature \cite{CNX_12, Wu_Zheng_Multigrid} and follows from our proof with only minor modifications.
	\begin{proof}[Proof of Theorem~\ref{theorem:solver}, connection of solver and estimator~\eqref{equation:pythagoras}]
	The proof consists of two steps.
	\medskip

	\noindent{\bf Step~1.} We show that there holds the identity
	\begin{align} \label{equation:int_res_LB_theorem}
		\begin{split}
			\enorm[\Big]{\sum_{\ell = 0}^{L-1} \lambda_\ell \rho_\ell}^2 - 2 \edual[\Big]{u_L^\star - v_L}{&\sum_{ \ell = 0}^{L-1} \lambda_\ell \rho_\ell} \\
			&= -\enorm{\rho_0}^2  + \sum_{\ell = 1}^{L-1} \enorm{\lambda_\ell \rho_\ell}^2
			- 2 \sum_{\ell = 1}^{L-1} \lambda_\ell \sum\limits_{z \in \VV_\ell^+} \enorm{\rho_{\ell,z}}^2.
		\end{split}
	\end{align}
	Indeed, note that $\sigma_\ell = \sum_{k=0}^\ell \lambda_k \rho_k$. By definition of the local lowest-order problems in \eqref{equation:alg_step1} and \eqref{equation:alg_step2} as well as the definition of $\rho_{\ell} = \sum_{z \in \VV_{\ell}^{+}} \rho_{\ell, z}$, we have
	\begin{align*}
		\edual[\Big]{u_L^\star - v_L}{\!\!\sum_{\ell = 0}^{L-1} \!\lambda_\ell \rho_\ell} ~&\eqreff*{equation:residual_functional}=~
		R_L(\rho_0)  + \!\!\sum_{\ell = 1}^{L-1} \lambda_\ell \!\! \sum\limits_{z \in \VV_\ell^+} \! \! R_L(\rho_{\ell,z}) ~\eqreff*{equation:alg_step1}=~  \enorm{\rho_0}^2  + \!\! \sum_{\ell = 1}^{L-1} \lambda_\ell \!\! \sum\limits_{z \in \VV_\ell^+} \! \! R_L(\rho_{\ell, z})
		\\
		&\eqreff*{equation:alg_step2}=~ \enorm{\rho_0}^2 + \sum_{\ell = 1}^{L-1} \lambda_\ell  \sum\limits_{z \in \VV_\ell^+} \Big(\enorm{\rho_{\ell,z}}^2 + \edual{\sigma_{\ell-1}}{\rho_{\ell,z}} \Big)
		\\
		&
		=\enorm{\rho_0}^2 + \sum_{\ell = 1}^{L-1} \lambda_\ell \sum\limits_{z \in \VV_\ell^+} \Big(\enorm{\rho_{\ell,z}}^2 + \sum_{k = 0}^{\ell-1} \edual{\lambda_k \rho_k}{\rho_{\ell,z}} \Big)
		\\
		&
		= \ \enorm{\rho_0}^2 + \sum_{\ell = 1}^{L-1} \lambda_\ell \sum\limits_{z \in \VV_\ell^+} \enorm{\rho_{\ell,z}}^2 + \sum_{\ell = 1}^{L-1} \sum_{k = 0}^{\ell-1} \edual{\lambda_k \rho_k}{\lambda_\ell \rho_\ell}.
	\end{align*}
	Thus, by expanding the square, we have
	\begin{align*}
		\enorm[\Big]{\sum_{\ell = 0}^{L-1} \lambda_\ell \rho_\ell}^2 - 2 \edual[\Big]{u_L^\star - &v_L}{\sum_{ \ell = 0}^{L-1} \lambda_\ell \rho_\ell}
		% 			\\
		% &
		= \sum_{\ell = 0}^{L-1}  \enorm{\lambda_\ell \rho_{\ell}}^2 - 2 \enorm{\rho_0}^2 - 2 \sum_{\ell = 1}^{L-1} \lambda_\ell \sum\limits_{z \in \VV_\ell^+} \enorm{\rho_{\ell,z}}^2 \\
		&= -\enorm{\rho_0}^2  + \sum_{\ell = 1}^{L-1} \enorm{\lambda_\ell \rho_\ell}^2
		- 2 \sum_{\ell = 1}^{L-1} \lambda_\ell \sum\limits_{z \in \VV_\ell^+} \enorm{\rho_{\ell,z}}^2.
	\end{align*}
	This proves the identity \eqref{equation:int_res_LB_theorem}.
	\medskip
	
	\noindent{\bf Step~2.} Recall that $\Phi(v_L) = v_L + \sigma_L = v_L + \sigma_{L-1} + \lambda_L \rho_L$. By definition of $R_L$ in \eqref{equation:residual_functional} and the choice of $\lambda_L$ in Algorithm~\ref{algorithm:solver}, we have
	\begin{align*}
		\enorm{u_L^\star - \Phi(&v_L)}^2
		= \enorm{u_L^\star - (v_L + \sigma_{L-1})}^2 - 2 \, \lambda_L \, \edual{u_L^\star - (v_L + \sigma_{L-1})}{\rho_L} + \enorm{\lambda_L\rho_L}^2
		\\
		&= \enorm{u_L^\star - (v_L+ \sigma_{L-1})}^2 -2\lambda_L \Big(R_L(\rho_{L}) - \edual{\sigma_{{L}-1}}{\rho_{L}}\Big) + \lambda_L \sum\limits_{z \in \VV_L} \enorm{\rho_{L,z}}^2
		\\
		&\eqreff*{equation:alg_step2}
		=~ \enorm[\Big]{u_L^\star - \Big(v_L + \sum_{\ell = 0}^{L-1} \lambda_\ell \rho_\ell \Big)}^2 - \lambda_L \sum\limits_{z \in \VV_L} \enorm{\rho_{L,z}}^2.
	\end{align*}
	For the first term it holds that
	\begin{align*}
		\enorm[\Big]{u_L^\star - \Big(v_L + \sum_{\ell = 0}^{L-1} \lambda_\ell \rho_\ell &\Big)}^2
		= \enorm{u_L^\star - v_L}^2 + \enorm[\Big]{\sum_{\ell = 0}^{L-1} \lambda_\ell \rho_\ell}^2 - 2 \edual[\Big]{u_L^\star - v_L}{\sum_{\ell = 0}^{L-1} \lambda_\ell \rho_\ell}
		\\
		&\eqreff*{equation:int_res_LB_theorem}
		=~ \enorm{u_L^\star - v_L}^2 -\enorm{\rho_0}^2  + \sum_{\ell = 1}^{L-1} \enorm{\lambda_\ell \rho_\ell}^2
		- 2 \sum_{\ell = 1}^{L-1} \lambda_\ell \sum\limits_{z \in \VV_\ell^+} \enorm{\rho_{\ell,z}}^2\\
		% \end{align*}
	% An application of \eqref{equation:loc_global_stepsizes} shows
	% \begin{align*}
		% 	\enorm{u_L^\star - v_L}^2 -\enorm{\rho_0}^2  + \sum_{\ell = 1}^{L-1} \enorm{\lambda_\ell \rho_\ell}^2
		% 	- &2 \sum_{\ell = 1}^{L-1} \lambda_\ell \sum\limits_{z \in \VV_\ell^+} \enorm{\rho_{\ell,z}}^2 \\
		&\eqreff*{equation:loc_global_stepsizes}
		\le \enorm{u_L^\star - v_L}^2 - \enorm{\rho_0}^2 -\sum_{\ell = 1}^{L-1} \lambda_\ell \sum\limits_{z \in \VV_\ell^+} \enorm{\rho_{\ell,z}}^2.
	\end{align*}
	Combining the last two estimates with the definition of $\etalg(v_L)$ in Algorithm~\ref{algorithm:solver}, we obtain
	\begin{align*}
		\enorm{u_L^\star - \Phi(v_L)}^2
		&\le \enorm{u_L^\exact - v_L}^2 - \enorm{\rho_0}^2 - \sum_{\ell = 1}^{L-1} \lambda_\ell \sum\limits_{z \in \VV_{\ell}^+} \enorm{\rho_{\ell,z}}^2 - \lambda_L \sum\limits_{z \in \VV_L} \enorm{\rho_{L,z}}^2\\
		&
		= \enorm{u_L^\star - v_L}^2 -  \etalg(v_L)^2.
	\end{align*}
	This concludes the proof of~\eqref{equation:pythagoras}.
	\end{proof}
	
	\begin{proof}[Proof of Theorem~\ref{theorem:solver}, lower bound in~\eqref{equation:equivalence}]
	The relation between the solver and the estimator given in~\eqref{equation:pythagoras} shows that $\etalg(v_L) \le \enorm{u_L^\star - v_L}$.
	\end{proof}
	
	\begin{proof}[Proof of Corollary~\ref{corollary:solver}, equivalence of~\eqref{equation:contraction} and~\eqref{equation:equivalence}]
	We prove that the solver contraction in~\eqref{equation:contraction} is equivalent to the upper bound of~\eqref{equation:equivalence}.
	
	First, suppose that~\eqref{equation:contraction} holds. Then, we proceed similarly as in the proof of~\eqref{equation:pythagoras} to obtain
	\begin{align*}
		&\enorm{u_L^\star -v_L }^2
		%\\
		%&
		= \enorm{u_L^\star -\Phi(v_L) }^2 - \enorm[\Big]{\sum_{\ell = 0}^{L-1} \lambda_\ell \rho_\ell}^2 \!\! \!+ 2 \, \edual[\Big]{u_L^\star - v_L}{\!\!\sum_{\ell = 0}^{L-1} \lambda_\ell \rho_\ell} \! + \lambda_L \!\! \sum\limits_{z \in \VV_L} \!\! \enorm{\rho_{L,z}}^2\!\!
		\\
		&\eqreff*{equation:int_res_LB_theorem}= \ \enorm{u_L^\star - \Phi(v_L)}^2 + \enorm{\rho_0}^2  - \sum_{\ell = 1}^{L-1} \enorm{\lambda_\ell \rho_\ell}^2
		+ 2 \sum_{\ell = 1}^{L-1} \lambda_\ell \sum\limits_{z \in \VV_\ell^+} \enorm{\rho_{\ell,z}}^2 + \lambda_L \sum\limits_{z \in \VV_L} \enorm{\rho_{L,z}}^2
		\\
		&
		\eqreff*{equation:contraction}\le \qctr^2 	\enorm{u_L^\star - v_L }^2 + 2\, \etalg(v_L)^2.
	\end{align*}
	Rearranging this estimate proves the upper bound in~\eqref{equation:equivalence} with $\crel^2 = 2/(1-\qctr^2) > 1$.
	
	Second, suppose the upper bound in~\eqref{equation:equivalence}. Then, it follows that
	\begin{align*}
		\enorm{u_L^\star - \Phi(v_L)}^2
		\reff{equation:pythagoras}\le \enorm{u_L^\star - v_L}^2 - \etalg (v_L)^2
		&\reff{equation:equivalence}\le \enorm{u_L^\star - v_L}^2 - \crel^{-2} \, \enorm{u_L^\star - v_L}^2.
	\end{align*}
	This verifies the solver contraction~\eqref{equation:contraction} for $\qctr^2 = 1- \crel^{-2} \in (0,1)$ and concludes the equivalence proof.
	\end{proof}

	\begin{proof}[Proof of Theorem~\ref{theorem:solver}, upper bound in~\eqref{equation:equivalence}]
	We use the stable decomposition of Proposition~\ref{lemma:multilevel_decomposition} on the algebraic error $ u_L^\star - v_L \in \V_L^p$ to obtain $v_0 \in \V_0^1, \   v_{\ell, z} \in \V_{\ell, z}^1$ and $v_{L,z} \in \V_{L,z}^p$ such that
	\begin{align}
		&u_L^\star - v_L = v_0 + \sum_{\ell=1}^{L-1}  \sum_{z \in \VV_{\ell}^{+}}  v_{\ell, z} + \sum_{z \in \VV_L}  v_{L, z} \notag
		\\
		&  \text{and} \quad \enorm{v_0}^2 + \sum_{\ell=1}^{L-1} \sum_{z \in \VV_{\ell}^{+}} \enorm{v_{\ell, z}}^2
		+ \sum_{z \in \VV_L} \enorm{v_{L, z}}^2
		\le  \csdsqu \enorm{u_L^\star-v_L}^2. \label{equation:error_decomposition_stability}
	\end{align}
	Note that $\sigma_\ell = \sum_{k=0}^\ell \lambda_k \rho_k$ for all $\ell = 0, \dots, L$; see Algorithm~\ref{algorithm:solver}. We use \eqref{equation:error_decomposition_stability} to develop
	\begin{align*}
		\enorm{u_L^\star-v_L}^2
		&=   \edual[\Big]{u_L^\star-v_L }{v_0 + \sum_{\ell=1}^{L-1}  \sum_{z \in \VV_\ell^+} v_{\ell, z}   +  \sum_{z \in \VV_L}  v_{L, z}}
		\\
		&
		\stackrel{\substack{\eqref{equation:residual_functional} \\ \eqref{equation:alg_step1}}}=    \edual{\rho_0}{v_0} + \sum_{\ell=1}^{L-1}  \sum_{z \in \VV_{\ell}^{+}}   R_L(v_{\ell, z}) +  \sum_{z \in \VV_L}  R_L(v_{L, z})\\
		&\eqreff{equation:alg_step2}= \! \edual{\rho_0}{v_0}
		\!  + \!\!  \sum_{\ell=1}^{L-1}  \! \sum_{z \in \VV_{\ell}^{+}}   \! \!  \big(
		\edual{\rho_{\ell, z}}{v_{\ell, z}}  + \edual{\sigma_{\ell-1}}{v_{\ell,z}}     \big)
		\\
		&\qquad
		+   \!\!  \sum_{z \in \VV_L}   \! \!  \big( \!
		\edual{\rho_{L, z}}{v_{L, z}}   +  \edual{\sigma_{L-1}}{v_{L,z}}  \big).
	\end{align*}
	Expanding $\sigma_\ell = \rho_0 + \sum_{k = 1}^{\ell} \lambda_\ell \rho_{\ell}$ and rearranging the terms finally leads to
	\begin{align*}
		\enorm{u_L^\star-v_L}^2
		&= \edual[\big]{\rho_0}{v_0 + \sum_{\ell=1}^{L-1}  \sum_{z \in \VV_{\ell}^{+}}  v_{\ell, z} +\sum_{z \in \VV_L}   v_{L, z} }
		+   \sum_{\ell=1}^{L-1}  \sum_{z \in \VV_{\ell}^{+}}   \!  \edual{\rho_{\ell, z}}{v_{\ell, z}}
		\\
		&  \quad
		+\!\! \sum_{z \in \VV_L} \! \edual{\rho_{L, z}}{v_{L, z}}
		+  \! \sum_{\ell=1}^{L-1}  \sum_{k=1}^{\ell-1} \edual[\Big]{\lambda_k \rho_k}{\!\!\!\sum_{z \in \VV_{\ell}^{+}}  \!\!v_{\ell, z}} +\! \sum_{k=1}^{L-1} \! \edual[\Big]{\lambda_k \rho_k}{\!\!\!\sum_{z \in \VV_L}  \!\! v_{L, z}}.\!\!
	\end{align*}
	Note that, until this point, only equalities are used. In the following, we will estimate each of the constituting terms of the algebraic error using Young's inequality in the form $ab \le (\alpha/2)\, a^2 + (2\alpha)^{-1}\, b^2$ with $\alpha = 4 \csd^{2}$, the strengthened Cauchy--Schwarz inequality, and patch overlap arguments as done in the proof of Lemma~\ref{lemma:norm_estimates}.
	Using the fact that $\lambda_0 =1$ and the decomposition of the error $u_L^\star-v_L = v_0 + \sum_{\ell=1}^{L-1}  \sum_{z \in \VV_{\ell}^{+}}  v_{\ell, z} + \sum_{z \in \VV_L}  v_{L, z}$, we see that the first term yields
	\begin{align*}
		\edual[\big]{\rho_0}{v_0 + \sum_{\ell=1}^{L-1}  \sum_{z \in \VV_{\ell}^{+}}  v_{\ell, z} + \! \! \sum_{z \in \VV_L}   v_{L, z} } \eqreff*{equation:error_decomposition_stability}=
		~\edual{\rho_0}{u_L^\star-v_L}
		\le
		\cfrac{1}{2}  \, \enorm{\lambda_0 \rho_0}^2
		+\cfrac{1}{2} \, \enorm{u_L^\star-v_L}^2.
	\end{align*}
	For the second term, we obtain that
	\begin{align*}
		\sum_{\ell=1}^{L-1} \sum_{z \in \VV_{\ell}^{+}}
		\edual{\rho_{\ell, z}}{v_{\ell, z}}
		& \le 2\csdsqu   \sum_{\ell=1}^{L-1} \sum_{z \in \VV_{\ell}^{+}}
		\enorm{\rho_{\ell, z}}^2 +   \frac{1 }{8 \csdsqu }
		\sum_{\ell=1}^{L-1}  \sum_{z \in \VV_{\ell}^{+}}
		\enorm{v_{\ell, z}}^2\\
		&\stackrel{\substack{\eqref{equation:bounds_stepsizes}}}\le
		2\csdsqu (d+1 )\sum_{\ell=1}^{L-1} \lambda_\ell \sum_{z \in \VV_{\ell}^{+}}
		\enorm{\rho_{\ell, z}}^2 +   \frac{1 }{8 \csdsqu }
		\sum_{\ell=1}^{L-1}  \sum_{z \in \VV_{\ell}^{+}}
		\enorm{v_{\ell, z}}^2,
	\end{align*}
	and similarly for the third term
	\begin{align*}
		\sum_{z \in \VV_L}
		\edual{\rho_{L, z}}{v_{L, z}}
		&\stackrel{\substack{\eqref{equation:bounds_stepsizes}}}\le
		2 \csdsqu  (d+1)\lambda_L \sum_{z \in \VV_L}
		\enorm{\rho_{L, z}}^2 +   \frac{1 }{8 \csdsqu}
		\sum_{z \in \VV_L}
		\enorm{v_{L, z}}^2.
	\end{align*}
	For the fourth term, we have
	\begin{align*}
		\sum_{\ell=1}^{L-1}  \sum_{k=1}^{\ell-1} \edual[\big]{\lambda_k \rho_k}{&\!\!\sum_{z \in \VV_{\ell}^{+}} v_{\ell, z}}
		\stackrel{\eqref{equation:strengthened_CS_adapt}}\le
		\cscs \Big( \sum_{k=1}^{L-2} \sum_{w \in \VV_{k}^{+}} \enorm{   \lambda_k \rho_{k,w}}^2 \Big)^{1/2} \Big( \sum_{\ell=1}^{L-1}\sum_{z \in \VV_{\ell}^{+}} \enorm{v_{\ell, z}}^2 \Big)^{1/2}\\
		&\le 2\cscs^2  \csdsqu \sum_{k=0}^{L-2}  \sum_{w \in \VV_{k}^{+}} \enorm{   \lambda_k \rho_{k,w}}^2
		+\frac{1 }{8 \csdsqu} \sum_{\ell=1}^{L-1}  \sum_{z \in \VV_{\ell}^{+}} \enorm{v_{\ell, z}}^2 \\
		&\stackrel{\eqref{equation:bounds_stepsizes}}\le 2\cscs^2  \csdsqu(d+1) \sum_{k=0}^{L-2}  \lambda_k \sum_{w \in \VV_{k}^{+}} \enorm{  \rho_{k,w}}^2
		+\frac{1 }{8 \csdsqu} \sum_{\ell=1}^{L-1}  \sum_{z \in \VV_{\ell}^{+}} \enorm{v_{\ell, z}}^2.
	\end{align*}
	Finally, to treat the last term where higher-order terms appear together with a sum over levels, we proceed similarly as in~\cite[Proof of Theorem 4.8]{CNX_12} and obtain
	\begin{align*}
		\sum_{k=1}^{L-1}  \edual[\Big]{\lambda_k \rho_k}{\sum_{z \in \VV_L} v_{L, z}}
		&=  \sum_{z \in \VV_L}    \edual[\Big]{\sum_{k=1}^{L-1} \lambda_k \rho_k}{v_{L, z}}\\
		&\le 2\csdsqu \sum_{z \in \VV_L}  \enorm[\Big]{  \sum_{k=1}^{L-1} \lambda_k   \rho_k}_{\omega_{L,z}}^2   +\frac{1 }{8 \csdsqu}  \sum_{z \in \VV_L} \enorm{v_{L, z}}^2.
	\end{align*}
	For the first term of the last bound, we have that
	\begin{align*}
		\sum_{z \in \VV_L}  \enorm[\Big]{  &\sum_{k=1}^{L-1} \lambda_k   \rho_k}_{\omega_{L,z}}^2 \lesssim  \enorm[\Big]{ \sum_{k=1}^{L-1} \lambda_k    \rho_k}^2 = \sum_{k=1}^{L-1}  \enorm{\lambda_k \rho_k}^2  +  2 \, \sum_{\ell=1}^{L-1} \sum_{k=1}^{\ell-1} \edual{\lambda_k \rho_k}{\lambda_\ell \rho_\ell} \\
		&\eqreff*{equation:strengthened_CS_adapt}\le \
		\sum_{k=1}^{L-1}  \enorm{\lambda_k \rho_k}^2  +  2
		\cscs \Big( \sum_{k=1}^{L-2}   \sum_{w \in \VV_{k}^{+}} \enorm{   \lambda_k \rho_{k,w}}^2 \Big)^{1/2} \Big( \sum_{\ell=1}^{L-1} \sum_{z \in \VV_{\ell}^{+}} \enorm{\lambda_\ell \rho_{\ell, z}}^2 \Big)^{1/2} \\
		&\stackrel{\substack{\eqref{equation:bounds_stepsizes} \\ \eqref{equation:loc_global_stepsizes}}}\le 
		\big( 1+ 2 \, \cscs(d+1) \big) \, \Bigl( \sum_{\ell=1}^{L-1}    \lambda_\ell
		\sum_{z \in \VV_\ell^+}     \enorm{  \rho_{\ell,z}}^2 \Bigr).
	\end{align*}
	Summing all the estimates of the algebraic error components and defining the constant $\crel^2 \coloneqq \max \{1/2, \csdsqu(d+1)\big( 2 + \cscs^2 + 2 \, \cscs(d+1)^{1/2}\big)\}$, we see that
	\begin{align*}
		\enorm{u_L^\star-v_L}^2
		&\le
		\cfrac{1}{2} \,  \enorm{\lambda_0 \rho_0}^2
		+\cfrac{1}{2} \, \enorm{u_L^\star-v_L}^2+ \! 4 \crel^2  \Big(  \sum_{\ell=1}^{L-1}   \lambda_\ell \! \!
		\sum_{z \in \VV_\ell^+}   \!  \enorm{  \rho_{\ell,z}}^2 +
		\lambda_L \! \!
		\sum_{z \in \VV_L}   \! \!   \enorm{  \rho_{L,z}}^2 \Big) \\
		&\qquad +  \frac{1}{4 \csdsqu} \Big(
		\sum\limits_{\ell=1}^{L-1}  \sum\limits_{z \in \VV_{\ell}^{+}} \enorm{v_{\ell, z}}^2 + \sum_{z \in \VV_L} \enorm{v_{L, z}}^2 \Big)
		\\
		&\stackrel{\eqref{equation:error_decomposition_stability}}\le
		4 \crel^2  \etalg(v_L)^2 + \cfrac{3}{4} \, \enorm{u_L^\star-v_L}^2.
	\end{align*}
	After rearranging the terms, we finally obtain that
	\begin{align}\label{equation:lower_bound_estimator}
		\enorm{u_L^\star-v_L}^2
		&\le \crel^2 \, \etalg(v_L)^2.
	\end{align}
	This proves the upper bound of \eqref{equation:equivalence} and thus concludes the proof of Theorem~\ref{theorem:solver}.
	\end{proof}
	
	%%-----------------------------
	%%      your bibliography
	%%-----------------------------
	\printbibliography

\end{document}